\documentclass{siamart171218}
\usepackage{lipsum}
\usepackage{amsmath}
\usepackage{amsfonts}
\usepackage{amssymb}
\usepackage{bm}
\usepackage{multirow}
\usepackage{multicol}
\usepackage{nicefrac}
\usepackage{diagbox}
\usepackage{graphicx}
\usepackage{epstopdf}
\usepackage{algorithmic}
\usepackage{tikz}
\usepackage{ulem}
\usetikzlibrary{arrows,shapes,backgrounds,matrix,positioning,calc,intersections}
\usetikzlibrary{patterns}
\ifpdf%
  \DeclareGraphicsExtensions{.eps,.pdf,.png,.jpg}
\else
  \DeclareGraphicsExtensions{.eps}
\fi
\usepackage{amsopn}

\usepackage{booktabs}

\newcommand{\RR}{\mathbb{R}}

\DeclareMathOperator{\ddiv}{div}
\newcommand{\Nedelec}{N\'{e}d\'{e}lec }
\newcommand{\vertiii}[1]{{\left\vert\kern-0.25ex\left\vert\kern-0.25ex\left\vert #1
        \right\vert\kern-0.25ex\right\vert\kern-0.25ex\right\vert}}

\newcommand{\edit}[1]{\textcolor{black}{{#1}}}

\newcommand{\AuE}{A_{\bm u}^E}
\newcommand{\BuE}{B_{\bm u}^E}
\newcommand{\ApE}{A_p^E}
\newcommand{\ApEs}{A_p^{E*}}

\newtheorem*{remark}{Remark}

\newcommand{\TheTitle}{%
  Robust Preconditioners for a New Stabilized Discretization of the Poroelastic Equations
}

\newcommand{\TheShortTitle}{%
    Robust Preconditioners for Poroelastic Equations
}

\newcommand{\TheName}{%
  P. B. Ohm,
  J. H. Adler,
  F. J. Gaspar,
  X. Hu,
  C. Rodrigo,
  L. T. Zikatanov
}

\newcommand{\TheFunding}{%
  \textbf{Funding:} The work of Adler, Hu, and Ohm is partially supported by the National Science Foundation under grant DMS-1620063\@, and the work of Zikatanov is supported in part by NSF DMS-1720114 and DMS-1819157.
  The work of Gaspar and Rodrigo is supported in part by the Spanish project FEDER/MCYT MTM2016-75139-R and the Diputaci\'on General de Arag\'on (Grupo de referencia APEDIF, ref. E24\_17R).
}

\author{
	J. H. Adler\thanks{Department of Mathematics, Tufts University, Medford, MA 02155
		(peter.ohm@tufts.edu, james.adler@tufts.edu, \hbox{xiaozhe.hu@tufts.edu}).}
	\and
	F. J. Gaspar\thanks{Departamento de Matem\'{a}tica Aplicada, IUMA, Universidad de Zaragoza, Zaragoza, Spain (fjgaspar@unizar.es, carmenr@unizar.es).}
	\and
	X. Hu\footnotemark[2]
	\and
	P. Ohm\footnotemark[2]
       \and
	C. Rodrigo\footnotemark[3]
	\and
	L. T. Zikatanov\thanks{Department of Mathematics, The Pennsylvania State University, University Park, PA 16802 (\hbox{ludmil@psu.edu}).}
}
\title{{\TheTitle}\thanks{\TheFunding}}
\headers{\TheShortTitle}{\TheName}
\ifpdf%
\hypersetup{%
  pdftitle={\TheTitle},
  pdfauthor={\TheName}
}
\fi

\begin{document}

\maketitle

\vspace{1cm}

\begin{abstract}
In this paper, we present block preconditioners for a stabilized
discretization of the poroelastic equations developed in
\cite{Rodrigo2017}. The discretization is proved to be well-posed with
respect to the physical and discretization parameters, and thus
provides a framework to develop preconditioners that are robust with
respect to such parameters as well. We construct both
  norm-equivalent (diagonal) and field-of-value-equivalent
  (triangular) preconditioners for both the stabilized discretization
  and a perturbation of the stabilized discretization that leads to a
  smaller overall problem after static condensation. Numerical tests for both two- and
  three-dimensional problems confirm the robustness of the block preconditioners with respect to the physical and discretization parameters.

\end{abstract}


\section{Introduction}\label{sec:intro}

In this work, we study the quasi-static Biot model for soil consolidation, where we assume a porous medium to be linearly elastic, homogeneous, isotropic, and saturated by an incompressible Newtonian fluid. According to Biot's theory \cite{biot1}, the consolidation process satisfies the following system of partial differential equations (PDEs):
\begin{eqnarray*}
\mbox{\rm equilibrium equation:}\quad & & -{\rm div} \, {\boldsymbol
\sigma}' +
\alpha \nabla \, p = \rho {\bm g}, \quad {\rm in} \, \Omega, \label{eq11} \\
\mbox{\rm constitutive equation:}\quad & & \bm{\sigma}' = 2\mu {\boldsymbol \varepsilon}(\bm{u})  + \lambda\ddiv(\bm{u}) {\bm I}, \quad {\rm in} \, \Omega,
\label{eq12} \\
\mbox{\rm compatibility condition:}\quad & & {\boldsymbol \varepsilon}({\bm u}) = \frac{1}{2}(\nabla {\bm u} + \nabla
{\bm u}^t), \quad {\rm in} \, \Omega,
\label{eq13} \\
\mbox{\rm Darcy's law:}\quad & & {\bm w} = - \displaystyle \frac{1}{\mu_f}{\bm K} ( \nabla
p - \rho_f {\bm g}), \quad {\rm in} \, \Omega,
\label{eq14} \\
\mbox{\rm continuity equation:}\quad & & \displaystyle \frac{\partial}{\partial t}\left( \frac{1}{M}p +
\alpha \ {\ddiv} \, {\bm u}\right)  + {\ddiv} \, {\bm w} = f, \quad {\rm in} \, \Omega,
\label{eq15}
\end{eqnarray*}
where $\lambda$ and $\mu$ are the Lam\'e coefficients, $\alpha$ is the Biot-Willis constant, $M$ is the bulk modulus, $\bm K$ is the absolute permeability tensor of the porous medium, $\mu_f$ is the viscosity of the fluid, ${\bm I}$ is the identity tensor, ${\bm u}$ is the displacement vector, $p$ is the pore pressure, ${\bm \sigma'}$ and ${\bm \epsilon}$ are the effective stress and strain tensors for the porous medium, and ${\bm w}$ is the percolation velocity, or Darcy's velocity, of the fluid relative to the soil. The right-hand term ${\bm g}$ is the density of applied body forces and the source term $f$ represents a forced fluid extraction or injection process. We consider a bounded open subset $\Omega \subset \RR^d$, $d=2,3$ with regular boundary $\Gamma$.

In many physical applications, the values of some of the parameters described above may vary over orders of magnitude.
For instance, in geophysical applications, the permeability can typically range from $10^{-9}$ to $10^{-21}m^2$
\cite{LeeMardalWinther,Wang2000}.
Similarly, in biophysical applications such as in the modeling of soft tissue or bone, the permeability can range from $10^{-14}$ to $10^{-16}m^2$
\cite{BenHatira2012,SmithHumphrey2007,Stoverud2016}.
The Poisson ratio, which is the ratio of transverse strain to axial strain,
ranges from $0.1$ to $0.5$ in these applications as well.
A Poisson ratio of $0.5$ indicates an incompressible material, at which
\edit{the linear-elastic term becomes positive semi-definite.}
Due to the variation in relevant  values of these physical parameters, it
is important to use discretizations that are stable, independently of
the parameters.
Therefore, in this work we build upon a parameter-robust
discretization introduced in \cite{Rodrigo2017}.

There are several formulations of Biot's model, and many stable
  finite-element schemes have been developed for each of them.  For
  instance, in what is called the two-field formulation (displacement
  and pressure are unknowns), Taylor-Hood elements which satisfy an
  appropriate inf-sup condition have been used \cite{MuradLoula92,
    MuradLoula94, MuradLoulaThome}. Unstable finite-element pairs
with appropriate stabilization techniques, such as the MINI element
have also been developed \cite{RGHZ2016}. Robust block preconditioners
for the two-field formulation were studied in \cite{Adler2017}. For
three-field formulations (displacement, pressure, and Darcy velocity
are unknowns), a parameter-independent approach is found in
\cite{HongKraus}.  There, the parameter-robust stability is studied
based on a slightly different norm used here, and robust block
diagonal preconditioners are proposed.   Another stable discretization
for the three-field formulation is Crouzeix-Raviart for displacement,
lowest order Raviart-Thomas-\Nedelec elements for Darcy's velocity,
and piecewise constants for the pressure \cite{Hu2017}.  Additionally,
a different three-field formulation (displacement, fluid pressure, and
total pressure are unknowns) was introduced in~\cite{LeeMardalWinther}
and a corresponding parameter-robust scheme is studied in the
same paper.  For a four-field formulation (stress tensor, fluid flux,
displacement, and pore-pressure are unknowns), a stable discretization was developed in \cite{lee_four_field}.

In all the cases above, typical discretizations result in a large-scale linear system of equations to solve at each time step. Such linear systems are usually ill-conditioned and difficult to solve in practice.
Also, due to their size, iterative solution
techniques are usually considered.
One approach to solving the coupled poromechanics equations considered
here is a sequential method, \edit{such as} the fixed stress iteration, which consists of first approximating
\edit{the fluid part and then the geomechanical part.}
This is then repeated until the solution has
converged to within a specified tolerance (see
\cite{ALMANI2016180,BAUSE2017745,BORREGALES20191466,BOTH2017101,Mikeli2013} for details). 
Another approach, is to solve the linear system simultaneously for all unknowns.
Examples of this in poromechanics can be found in
\cite{Adler2017,CASTELLETTO2016894,NME:NME2702,NLA:NLA372,GASPAR2017526,NLA:NLA2074} and the references within.
Analysis from \cite{Castelletto2015,WHITE201655} indicates that such a fully-implicit method outperforms the convergence rate of the sequential-implicit methods.

Thus, in this work, we take the latter approach and develop robust
block preconditioners (e.g. \cite{HElman_etal_2006a,HElman_etal_2008a,TGeenen_etal_2010a}) to accelerate
the convergence of Krylov subspace methods solving the full linear system of equations resulting from the
discretization of a three-field formulation of Biot's model.
The proposed preconditioners take advantage of the block structure of
the discrete model, decoupling the different fields at the
preconditioning stage. Such block preconditioning is primarily attractive
due to its simplicity, which allows us to focus on the character of
the diagonal blocks, and to leverage extensive work on solving simpler
problems.   For instance, one can take advantage of algebraic
multigrid for some of the blocks \cite{ABrandt_SFMcCormick_JWRuge_1984a},
or auxiliary space decomposition for others \cite{Arnold2000,HiptmairXu2007}.
Finally, since we use a stabilized discretization that is well-posed
with respect to the physical and discretization parameters \cite{Rodrigo2017}, we are
able to develop robust block preconditioners that efficiently solve the
linear systems, independently of such parameters as well.

The rest of the paper is organized as follows. Section~\ref{sec:threefield}
reintroduces the three-field formulation and stabilized finite-element
discretization considered.
Stability of a perturbation to the finite-element discretization is discussed in Section~\ref{sec:wellposed}.
The block preconditioners are then developed in
Section~\ref{sec:blockprec}, presenting both block diagonal and block
triangular approaches.  Finally, numerical results confirming
the robustness and effectiveness of the preconditioners are shown
in Section~\ref{sec:num}, and concluding remarks are made in Section~\ref{sec:conc}.

\section{Three-Field Formulation and its Discretization}\label{sec:threefield}


The focus of this paper is on the three-field formulation, in which Darcy's velocity, ${\bm w}$, is also a primary unknown in addition to the displacement, ${\bm u}$, and pressure, $p$. As a result we have the following system of PDEs:
\begin{eqnarray*}
&& -\ddiv\bm{\sigma}' + \alpha \nabla p = \rho \bm{g},\qquad
\mbox{where} \quad
  \bm{\sigma}' = 2\mu  {\boldsymbol \varepsilon}(\bm{u})  +
   \lambda\ddiv(\bm{u}) \bm{I}, \label{three-field1}\\
  && \displaystyle \frac{\partial}{\partial t}\left( \frac{1}{M}p + \alpha \ {\ddiv} \, {\bm u}\right)  + {\ddiv} \, {\bm w} = f,
\label{three-field3}\\
&& {\bm K}^{-1}\mu_f {\bm w} + \nabla p = \rho_f {\bm g}. \label{three-field2}
\end{eqnarray*}
This system is often subject to the following set of boundary conditions. \edit{Though non-homogeneous boundary conditions can also be used, for the sake of simplicity, we consider homogeneous case in this work}:
\begin{eqnarray*}
&&  p = 0, \quad \mbox{for}\quad x\in\overline{\Gamma}_t,
\quad \boldsymbol \sigma' \, {\bm n} = {\bm 0}, \quad
\quad \mbox{for}\quad x\in\Gamma _t, \label{bound-cond-t}\\
&& {\bm u} = {\bm 0}, \quad  \mbox{for}\quad x\in \overline{\Gamma}_c,\quad
\bm{w}\cdot \bm{n}=0, \quad \mbox{for}\quad x\in \Gamma_c,
\label{bound-cond-c}
\end{eqnarray*}
where ${\bm n}$ is the outward unit normal to the boundary,
$\overline{\Gamma}=\overline \Gamma_t \cup \overline \Gamma_c$;
$\Gamma_t$ and $\Gamma_c$ are open (with respect to $\Gamma$) subsets of
$\Gamma$ with nonzero measure. The initial condition at $t=0$ is given by,
\begin{equation*}\label{ini-cond}
      \left( \frac{1}{M}p + \alpha \ddiv {\bm u} \right)\, (\bm{x},0)=0, \,  \bm{x} \in\Omega.
\end{equation*}
This yields the following mixed formulation
for Biot's three-field consolidation model:\\
For each $t\in (0,T]$, find 
$(\bm u(t), p(t), \bm w(t))\in \bm V \times Q \times \bm W$ such that
\begin{eqnarray}
  && a(\bm{u},\bm{v}) - (\alpha p,\ddiv \bm{v}) = (\rho\bm g,\bm{v}),
     \quad \forall \  \bm{v}\in \bm V, \label{variational1}\\
   &&  \left(\frac{1}{M}\frac{\partial p}{\partial t} ,q \right)+\left(\alpha\ddiv \frac{\partial \bm{u}}{\partial t},q\right) + (\ddiv \bm{w},q)   = (f,q), \quad \forall \ q \in Q,\label{variational3}\\
  && ({\bm K}^{-1}\mu_f\bm{w},\bm{r}) - (p,\ddiv \bm{r}) = (\rho_f {\bm g}, \bm{r}), \quad \forall \ \bm{r}\in \bm W,\label{variational2}
\end{eqnarray}
where,
\begin{equation}\label{bilinear}
a(\bm{u},\bm{v}) =
2\mu\int_{\Omega}{\boldsymbol \varepsilon}(\bm{u}):{\boldsymbol \varepsilon}(\bm{v}) +
\lambda\int_{\Omega} \ddiv\bm{u}\ddiv\bm{v},
\end{equation}
corresponds to linear elasticity and $(\cdot,\cdot)$ denotes the standard inner product on $L^2(\Omega)$. The function spaces used in the
variational form are
\begin{eqnarray*}
&&{\bm V} = \{{\bm u}\in {\bm H}^1(\Omega) \ |  \ {\bm
   u}|_{\overline{\Gamma}_c} = {\bm 0} \},\\
  &&Q = L^2(\Omega),\\
&&{\bm W} = \{{\bm w} \in \bm{H}(\ddiv,\Omega) \ | \  ({\bm w}\cdot {\bm n})|_{\Gamma_c} = 0\},
\end{eqnarray*}
where ${\bm H}^1(\Omega)$ is the space of square integrable
vector-valued functions whose first derivatives are also square
integrable, and $\bm{H}(\ddiv,\Omega)$ contains the square integrable
vector-valued functions with square integrable divergence.

In \cite{Rodrigo2017}, we developed a stabilized discretization for the three-field formulation described above.
Given a partition of the domain, $\Omega$, into
$d$-dimensional simplices,
$\mathcal T_h$, we associate a triple of piecewise
polynomial, finite-dimensional spaces,
\begin{equation*}\label{include}
\bm{V}_h\subset \bm{V}, \quad
Q_h \subset Q, \quad \bm{W}_h\subset \bm{W}.
\end{equation*}
More specifically, if we choose a piecewise linear continuous
finite-element space, $\bm{V}_{h,1}$, enriched with edge/face (2D/3D)
bubble functions, $\bm{V}_b$, to form $\bm{V}_h = \bm{V}_{h,1} \oplus
\bm{V}_b$ (see~\cite[pp. 145-149]{GR1986}), a lowest order
Raviart-Thomas-\Nedelec space (RT0) for $\bm{W}_h$, and a piecewise
constant space (P0) for $Q_h$, Stokes-Biot stable conditions described in Section~\ref{sec:wellposed} are satisfied and the
formulation is well-posed.
Then, using backward Euler as a time discretization on a
time interval $(0,t_{\max{}}]$ with constant time-step size 
$\tau$, the discrete scheme corresponding to the three-field
formulation~\eqref{variational1}-\eqref{variational2} reads:\\
Find $(\bm u_h^m, p_h^m, \bm w_h^m)\in \bm V_h \times Q_h \times \bm W_h$
such that
\begin{eqnarray*}
  && a(\bm{u}_h^m,\bm{v}_h) - (\alpha p_h^m,\ddiv \bm{v}_h) = (\rho\bm g,\bm{v}_h),
     \quad \forall \  \bm{v}_h\in \bm V_h, \label{discrete1}\\
   &&  \left(\frac{1}{M} p^m_h ,q_h \right)+\left(\alpha\ddiv \bm{u}_h^m,q_h\right) + \tau(\ddiv \bm{w}_h^m,q_h)   = (\widetilde{f}, q_h), \quad \forall \ q_h \in Q_h,\label{discrete3}\\
  && \tau({\bm K}^{-1}\mu_f\bm{w}_h^m,\bm{r}_h) - \tau(p_h^m,\ddiv \bm{r}_h) = \tau(\rho_f {\bm g}, \bm{r}_h), \quad \forall \ \bm{r}_h\in \bm W_h,\label{discrete2}
\end{eqnarray*}
where
$(\widetilde{f}, q_h) =\tau(f,q_h) + \left(\frac{1}{M} p^{m-1}_h ,q_h
\right)+\left(\alpha\ddiv \bm{u}_h^{m-1},q_h\right)$,
and
$(\bm{u}_h^m, p_h^m, \bm{w}_h^m) $ is an approximation to
$\left(\bm{u}(\cdot, t_m), p(\cdot, t_m), \bm{w}(\cdot, t_m)\right), $ at time
$t_m = m\tau, \ m = 1,2,\ldots.$
\edit{The last equation has been scaled by $\tau$ for symmetry.  To simplify the notation, we carry out the following stability analysis for a constant time-step size.  However, we note that utilizing a variable time-step size leads to analogous results.}
Moreover, this discrete variational form can be represented in block matrix form,
\begin{equation}\label{block_form}
\mathcal{ A} \left(
\begin{array}{c}
{\bm u}_h^b \\
{\bm u}_h^l \\
   p_h \\
  {\bm w}_h
\end{array}
\right) =
{\bm b}, \ \ \hbox{with} \ \
\mathcal{ A} = \left(
\begin{array}{cccc}
A_{bb} & A_{bl} & \alpha B_b^T & 0\\
A_{bl}^T & A_{ll} & \alpha B_l^T &0 \\
  -\alpha B_b & -\alpha B_l & \frac{1}{M} M_p & -\tau B_{\bm w} \\
  0 & 0 & \tau B_{\bm w}^T & \tau M_{\bm w}
\end{array}
\right),
\end{equation}
where ${\bm u}_b$, ${\bm u}_l$, $p$, and ${\bm w}$ are the unknown vectors for the bubble components of the displacement, the
piecewise linear components of the displacement, the
pressure, and the Darcy velocity, respectively. The
blocks in the definition of matrix $\mathcal{A}$ correspond to the following
bilinear forms:
\begin{eqnarray*}
  && a(\bm{u}_h^b,\bm{v}_h^b) \rightarrow A_{bb}, \quad a(\bm{u}_h^l,\bm{v}_h^b) \rightarrow A_{bl}, \quad  a(\bm{u}_h^l,\bm{v}_h^l) \rightarrow A_{ll}, \\
   &&-( \ddiv \bm{u}_h^b, q_h) \rightarrow B_b,  \quad   -(
   \ddiv \bm{u}_h^l, q_h) \rightarrow B_l, \quad -(\ddiv \bm{w}_h,q_h) \rightarrow B_{\bm w},\\
  && ({\bm K}^{-1}\mu_f\bm{w}_h,\bm{r}_h) \rightarrow  M_{\bm w}, \quad
 \left( p_h ,q_h \right) \rightarrow M_p,
\end{eqnarray*}
where ${\bm u}_h = {\bm u}_h^l+{\bm u}_h^b$,
${\bm u}_h^l\in \bm{V}_{h,1}$, ${\bm u}_h^b\in \bm{V}_b$, and an analogous
decomposition for $\bm{v}_h$.  We further define two matrices for use later,
\begin{equation*}
a(\bm{u}_h, \bm{v}_h) \rightarrow A_{\bm{u}}, \quad -(\ddiv \bm{u}_h, q_h) \rightarrow B_{\bm{u}},
\end{equation*}
such that $A_{\bm u} = \left( \begin{array}{cc} A_{bb} & A_{bl}\\A_{lb} & A_{ll}  \end{array}\right)$ and $B_{\bm u} = \begin{pmatrix} B_{b} & B_{l} \end{pmatrix}$.

A noteworthy result of \cite{Rodrigo2017} is that one can replace the enrichment bubble block, $A_{bb}$, in (\ref{block_form}) with a spectrally equivalent diagonal matrix, $D_{bb} := (d+1)\text{diag}(A_{bb})$,
resulting in the following linear operator,
\begin{equation}\label{block_form_diag}
\mathcal{A}^D = \left(
\begin{array}{cccc}
D_{bb} & A_{bl} & \alpha B_b^T & 0\\
A_{bl}^T & A_{ll} & \alpha B_l^T &0 \\
  -\alpha B_b & -\alpha B_l & \frac{1}{M} M_p & -\tau B_{\bm w} \\
  0 & 0 & \tau B_{\bm w}^T & \tau M_{\bm w}
\end{array}
\right).
\end{equation}
Not only is the resulting operator sparser than the operator
in (\ref{block_form}), the stabilization term can be eliminated  from
the operator in a straightforward way (i.e., static condensation), yielding,
\begin{equation}\label{block_form_elim}
\mathcal{A}^E =
\left(
  \begin{array}{ccc}
  A_{ll}-A_{bl}^T D_{bb}^{-1}A_{bl} & \alpha B_l^T- \alpha A_{bl}^T D_{bb}^{-1}B_b^T & 0\\
    -\alpha B_l+ \alpha B_b D_{bb}^{-1} A_{bl} & \frac{1}{M} M_p+\alpha^2 B_b D_{bb}^{-1} B_{b}^T & -\tau B_{\bm w} \\
  0 & \tau B_{\bm w}^T & \tau M_{\bm w}
  \end{array}
  \right).
\end{equation}
Thus, we obtain an optimal stable discretization with the lowest possible
number of degrees of freedom, equivalent to a discretization with P1-RT0-P0 elements, which itself is not stable \cite{Rodrigo2017}.
\edit{While we have reduced the number of degrees of freedom, we note that the sparsity structure of the stiffness matrix has changed as well.}
\edit{The number of non-zeros added to each row depends on the structure of the mesh. The (1,1) block and the (2,1) block increase in non-zeros per row by the number of elements adjacent to a vertex times dimension. In the worst case scenario for the structured grid formed by division of cubes into tetrahedrons, the (1,1) block grows from 37 non-zeros per row to 81 non-zeros per row. The (1,2) block and (2,2) block increase in non-zeros per row by the number of elements adjacent to each element. For the (1,2) block, this doubles the non-zeros per row. The (2,2) block is originally diagonal so the non-zeros per row increases to (spatial) dimension+2. In all cases, the computational cost of multiplication by the modified matrix has the same asymptotic behavior as the mesh size approaches zero.}

In \cite{Rodrigo2017}, it is discussed that due to the spectral
equivalence between $A_{bb}$ and $D_{bb}$, the formulation
(\ref{block_form_diag}) is still well-posed, and remains
well-posed independently of
the physical and discretization parameters.
In the following section \edit{(and appendices)}, we show this in detail, and prove that formulation (\ref{block_form_elim}) is also well posed independently of the physical and discretization parameters.

\section{Well-Posedness}\label{sec:wellposed}
The well-posedness of the discretized system provides a convenient framework with which to construct block preconditioners. The discrete system using bubble enriched P1-RT0-P0 finite elements (\ref{block_form}), which will be referred to as the ``full bubble system'', is shown to be well-posed in \cite{Rodrigo2017}.
However, since (\ref{block_form}) is indefinite, the well-posedness of (\ref{block_form_elim}) does not simply follow directly.
Therefore, in this section, we show well-posedness of (\ref{block_form_elim}), as well as (\ref{block_form_diag}), which enables block preconditioners for both the full system and the ``bubble-eliminated system" to be constructed using the same framework.  \edit{Since the proofs are quite technical, we include them in the appendices for completeness.}
First, we give a short overview of the full bubble system case.

To start, for any symmetric positive definite (SPD) matrix $H$, we define the corresponding inner product as $(\bm{u}, \bm{v})_H := (H \bm{u}, \bm{v})$ and the induced norm as\\ $\| \bm{v} \|_H^2:= (\bm{v}, \bm{v})_H$.  In association with the discretized space, $\bm{X}_h:= \bm{V}_h \times Q_h \times \bm{W}_h$,
we introduce the following weighted norm, for $\bm{x}_h = (\bm{u}_h, p_h, \bm{w}_h)^T \in \bm{X}_h$,
\begin{equation}\label{weighted-norm}
\vertiii{\bm{x}_h} := \left [  \| {\bm u}_h \|^2_{A_{\bm u}} +c_p^{-1}  \| p_h \|_{M_p}^2 + \tau \| {\bm w}_h \|^2_{M_{\bm w}} + \tau^2 c_p \| B_{\bm w} {\bm w}_h\|_{M_p^{-1}}^2 \right ]^{1/2},
\end{equation}
where $c_p := \left(\frac{\alpha^2}{\zeta^2}+\frac{1}{M}\right)^{-1}$ with $\zeta := \sqrt{\lambda+2\mu/d}$,
and $d=2$ or $3$ is the dimension of the problem.
Under certain conditions (referred to
as Stokes-Biot stability \cite[Def. 3.1]{Rodrigo2017}) on the space $\bm{X}_h$, the block matrix form $\mathcal{A}$ defined in~\eqref{block_form} is well-posed with respect to the weighted norm (\ref{weighted-norm}), i.e., the following
continuity and inf-sup condition hold for $\bm{x}_h\in \bm{X}_h$ and $\bm{y}_h = (\bm{v}_h, q_h, \bm{r}_h)^T \in \bm{X}_h$,
\begin{align}
& \sup_{\bm{0} \neq \bm{x}_h \in \bm{X}_h} \sup_{\bm{0} \neq \bm{y}_h \in \bm{X}_h } \frac{ (\mathcal{A} \bm{x}_h, \bm{y}_h)}{ \vertiii{\bm{x}_h} \vertiii{\bm{y}_h} } \leq \varsigma, \label{supsup-B} \\
& \inf_{\bm{0} \neq \bm{y}_h \in \bm{X}_h} \sup_{\bm{0} \neq \bm{x}_h \in \bm{X}_h} \frac{ (\mathcal{A} \bm{x}_h, \bm{y}_h)}{ \vertiii{\bm{x}_h} \vertiii{\bm{y}_h} } \geq \gamma, \label{infsup-B}
\end{align}
with constants $\varsigma >0$ and $\gamma > 0$ independent of mesh size $h$, time step size $\tau$, and the physical parameters.
As mentioned earlier, these conditions are satisfied by our choice of finite-element spaces.

System \eqref{block_form_diag} satisfies similar continuity and inf-sup conditions.
From \cite{Rodrigo2017}, we know that $A_{bb}$ is spectrally equivalent to $D_{bb}$ and $A_{\bm{u}}$ is spectrally equivalent to $A_{\bm{u}}^D = \begin{pmatrix} D_{bb} & A_{bl}\\A_{lb} & A_{ll}  \end{pmatrix}$, specifically,
\begin{equation}\label{eqn:ADequiv}
\| \bm{u}^b \|_{A_{bb}} \leq \| \bm{u}^b \|_{D_{bb}} \leq \eta \| \bm{u}^b \|_{A_{bb}} \ \text{and} \
\| \bm{u} \|_{A_{\bm{u}}} \leq \| \bm{u} \|_{A_{\bm{u}}^D} \leq \eta\| \bm{u} \|_{A_{\bm{u}}},
\end{equation}
where the constant $\eta$ depends only on the shape regularity of the mesh.
With the above results, we now \edit{state} 
the well-posedness of the system given by \eqref{block_form_diag}.

\begin{theorem}\label{thm:diag_wellposed}
If $(\bm{V}_h, \bm{W}_h, Q_h)$ is Stokes-Biot stable, then:
\begin{align}
& \sup_{\bm{0} \neq \bm{x}_h \in \bm{X}_h} \sup_{\bm{0} \neq \bm{y}_h \in \bm{X}_h } \frac{ (\mathcal{A}^D \bm{x}_h, \bm{y}_h)}{ \|\bm{x}_h\|_{\mathcal{D}} \|\bm{y}_h\|_{\mathcal{D}}} \leq \tilde{\varsigma}, \label{supsup-AD} \\
& \inf_{\bm{0} \neq \bm{y}_h \in \bm{X}_h} \sup_{\bm{0} \neq \bm{x}_h \in \bm{X}_h} \frac{ (\mathcal{A}^D \bm{x}_h, \bm{y}_h)}{ \|\bm{x}_h\|_{\mathcal{D}} \|\bm{y}_h\|_{\mathcal{D}} } \geq \tilde{\gamma}, \label{infsup-AD}
\end{align}
where,
\begin{equation*} 
\mathcal{D}  =
\begin{pmatrix}
  D_{bb}	& A_{bl}	& 0		& 0 \\
  A_{bl}^T	& A_{ll}	& 0		& 0 \\
  0		& 0 		& \left(\frac{\alpha^2}{\zeta^2} + \frac{1}{M}\right) M_p 	& 0 \\
  0 		& 0		& 0			& \tau M_{\bm w} + \tau^2 c_p A_{\bm w}
\end{pmatrix},
\end{equation*}
and $A_{\bm w} := B_{\bm w}^T M_p^{-1} B_{\bm w}$.
\end{theorem}

\begin{remark}
In general, Theorem~\ref{thm:diag_wellposed} implies that if we have a well-posed saddle point problem with an SPD first diagonal block, one can replace the first diagonal block by a spectrally equivalent matrix and the resulting saddle point problem is still well-posed.
\end{remark}

\edit{The proof of Theorem~\ref{thm:diag_wellposed} follows from the framework presented in \cite{HongKraus}. We have included a proof in Appendix \ref{sec:appA}, as there are details related to the perturbed bilinear form that are not straightforward.}
Next, we consider the reduced bubble-eliminated formulation, (\ref{block_form_elim}).
\edit{Here, we denote $\bm{X}_h^{E} := {\bm V}_{1,h}\times Q_h\times \bm W_h$ the discretized finite-element space after bubble elimination.}

\begin{theorem}\label{thm:ElimWP} 
    If the full system (\ref{block_form}) is well-posed, satisfying \eqref{supsup-B} and \eqref{infsup-B} with respect to the norm~\eqref{weighted-norm}, then the bubble-eliminated system, \eqref{block_form_elim},
    satisfies the following inequalities for $\bm{x}^E = (\bm{u}_l, p_h, \bm{w}_h)^T \in \bm{X}^E_h$ and $\bm{y}^E = (\bm{v}_l, q_h, \bm{r}_h)^T \in \bm{X}^E_h$,
    \begin{equation} \label{infsup-E}
	\inf_{\bm{0}\neq \bm{x}^E \in \bm{X}_h^E} \sup_{\bm{0} \neq \bm{y}^E\in \bm{X}_h^E }
	\frac{(\mathcal{A}^E \bm{x}^E, \bm{y}^E)}{\|\bm{x}^E\|_{\mathcal{D}^E}\|\bm{y}^E\|_{\mathcal{D}^E}} \geq \gamma^*,
    \end{equation}
    and,
    \begin{equation} \label{continuity-E}
      \sup_{\bm{0} \neq \bm{x}^E \in \bm{X}_h^E} \sup_{\bm{0} \neq \bm{y}^E \in \bm{X}^E_h} 	\frac{(\mathcal{A}^E \bm{x}^E, \bm{y}^E)}{\| \bm{x}^E \|_{\mathcal{D}^E} \| \bm{y}^E \|_{\mathcal{D}^E}} \leq \varsigma ,
    \end{equation}
    where
	\begin{equation*}
	  \mathcal{D}^E=
	  \left(\begin{array}{ccc}
	  A_{ll} - A_{bl}^T D_{bb}^{-1} A_{bl}		& 0				& 0 \\
	  0	 	& \alpha^2 B_b D_{bb}^{-1} B_b^T + c_p^{-1}M_p 	& 0 \\
	  0		& 0			& \tau M_{\bm w} + \tau^2 c_p A_{\bm w}
	  \end{array}\right),
	\end{equation*}
	with
	\begin{equation} \label{eqn:weighted_norm_elim}
\| \bm{x}^E \|_{\mathcal{D}^E}^2 =  (\mathcal{D}^E \bm{x}^E, \bm{x}^E).
	\end{equation}
Thus, (\ref{block_form_elim}) is well-posed with respect to the weighted norm (\ref{eqn:weighted_norm_elim}).
\end{theorem}
\edit{The proof of Theorem~\ref{thm:ElimWP} is technical due to $\mathcal{A}$ being indefinite.  Therefore, it is included in Appendix \ref{sec:appB} for the interested reader.}

\section{Block Preconditioners}\label{sec:blockprec}
Next, we use the properties of the well-posedness to develop block preconditioners for~$\mathcal{A}$ and $\mathcal{A}^E$.  Following the general framework developed \edit{in~\cite{CASTELLETTO2016894,Loghin2004,NLA:NLA716,WHITE201655,Hong2019}}, we first consider block diagonal preconditioners (also known as norm-equivalent preconditioners).  Then, we discuss
block triangular (upper and lower) preconditioners following the framework developed in~\cite{Klawonn1999,Loghin2004,Yicong2016,Starke1996} for Field-of-Value (FOV) equivalent preconditioners.  For both cases, we show that the theoretical bounds on their performance remain independent of the discretization and physical parameters of the problem.

\subsection{Block Diagonal Preconditioner}\label{sec:blockdiag}

Both the full bubble system, (\ref{block_form}), and the bubble-eliminated system, (\ref{block_form_elim}), are well-posed, satisfying inf-sup conditions, (\ref{infsup-B}) and (\ref{infsup-E}) respectively. Based on the framework proposed in~\cite{Loghin2004,NLA:NLA716},  a natural choice for a norm-equivalent preconditioner is the Riesz operator with respect to the inner product corresponding to respective weighted norm~(\ref{weighted-norm}) or (\ref{eqn:weighted_norm_elim}).

\subsubsection{Full Bubble System}
For the full bubble system, the Riesz operator for (\ref{weighted-norm}) takes the following block diagonal matrix form:
\begin{equation}\label{blk-diag-prec}
	\mathcal{B}_D =
	\left(
	\begin{array}{ccc}
	A_{\bm u} & 0 & 0 \\
	0 & \left(\frac{\alpha^2}{\zeta^2} + \frac{1}{M}\right) M_p &0 \\
	0 & 0 & \tau M_{\bm w} + \tau^2\left(\frac{\alpha^2}{\zeta^2} + \frac{1}{M}\right)^{-1} A_{\bm w}
	\end{array}
	\right)^{-1}.
\end{equation}

In practice, applying the preconditioner $\mathcal{B}_D$ involves the
action of inverting the diagonal blocks exactly, which is
expensive and sometimes infeasible.
Therefore, we replace the diagonal blocks by their spectrally-equivalent symmetric and positive definite approximations,
\begin{equation}\label{blk-diag-prec2}
	\widehat{\mathcal{B}_D} =
	\left(	\begin{array}{ccc}
	S_{\bm u} & 0 & 0 \\
	0 & S_{p} & 0\\
	0 & 0 & S_{\bm w}
	\end{array} \right).
\end{equation}
Here, $S_{\bm u}$, $S_{\bm w}$, and $S_p$ are spectrally equivalent to the action of the inverse of their respective diagonal blocks in $\mathcal{B}_D$,
\begin{eqnarray}
	&c_{1, \bm{u}}(S_{\bm u}\bm{u},\bm{u}) \leq (A_{\bm u}^{-1} \bm{u},\bm{u}) \leq c_{2,\bm{u}}(S_{\bm u} \bm{u},\bm{u}) \label{eqn:equivalentHu},\\
	&c_{1,p}(S_p p, p) \leq \left(\left(\frac{\alpha^2}{\zeta^2} + \frac{1}{M}\right)^{-1} M_p^{-1} p, p\right) \leq c_{2,p}(S_p p , p) \label{eqn:equivalentHp},\\
&\qquad
	c_{1,\bm{w}}(S_{\bm w}\bm{w},\bm{w}) \leq \left(\left(\tau M_{\bm w} + \tau^2\left(\frac{\alpha^2}{\zeta^2} + \frac{1}{M}\right)^{-1}A_{\bm w} \right)^{-1} \bm{w},\bm{w}\right)
		\leq c_{2,\bm{w}}(S_{\bm w} \bm{w},\bm{w}) \label{eqn:equivalentHw},
\end{eqnarray}
where the constants $c_{1,\bm{u}}$, $c_{1,p}$, $c_{1,\bm{w}}$,
$c_{2,\bm{u}}$, $c_{2,p}$, and $c_{2,\bm{w}}$ are independent of discretization and physical parameters.
\edit{In practice, $S_{\bm u}$ can be defined by standard multigrid methods.
For large values of $\tau$ the matrix $\tau M_{\bm w} + \tau^2\left(\frac{\alpha^2}{\zeta^2} + \frac{1}{M}\right)^{-1} A_{\bm w}$
is numerically close to singular and requires special preconditioners. With this in mind,
$S_{\bm w}$ can be defined by either an HX-preconditioner (Auxiliary Space Preconditioner)
\cite{HiptmairXu2007,KolevVassilevski} or multigrid with special block smoothers
\cite{Arnold2000}.}
\edit{In the case of heterogeneous coefficients, specialized approaches such as in \cite{KrausPrecondHdiv2016} can be used.}
In the full bubble case, $S_p$ is obtained by a diagonal scaling ($M_p$ is diagonal when using P0 elements).
Thus, $S_p = \left(\frac{\alpha^2}{\zeta^2} + \frac{1}{M}\right)^{-1} M_p^{-1}$.

\subsubsection{Bubble-Eliminated System}
In the bubble-eliminated case, the operator for (\ref{eqn:weighted_norm_elim}) takes the following block diagonal matrix form:
\begin{equation}\label{blk-diag-precE}
	\mathcal{B}_D^E =
	\left(
	\begin{array}{ccc}
	\AuE & 0 & 0 \\
	0 & \ApEs &0 \\
	0 & 0 & \tau M_{\bm w} + \tau^2\left(\frac{\alpha^2}{\zeta^2} + \frac{1}{M}\right)^{-1} A_{\bm w}
	\end{array}
	\right)^{-1}.
\end{equation}
Here,
$\AuE =  A_{ll}-A_{bl}^T D_{bb}^{-1}A_{bl}$ and
$\ApEs = \left(\frac{\alpha^2}{\zeta^2} + \frac{1}{M}\right) M_p+\alpha^2 B_b D_{bb}^{-1} B_{b}^T$.

Again, we replace the diagonal blocks by their spectrally-equivalent symmetric and positive definite approximations,
\begin{equation}\label{blk-diag-prec2E}
	\widehat{\mathcal{B}_D^E} =
	\left(	\begin{array}{ccc}
	S_{\bm u}^E & 0 & 0 \\
	0 & S_{p}^E & 0\\
	0 & 0 & S_{\bm w}
	\end{array} \right).
\end{equation}
Here, $S_{\bm u}^E$ and $S_p^E$ are spectrally equivalent to the action of the inverse of their respective diagonal blocks in,
${\mathcal{B}_D^E}$,
\begin{eqnarray}
	&c_{1, \bm{u}}^E(S_{\bm u}^E\bm{u},\bm{u}) \leq ((\AuE)^{-1} \bm{u},\bm{u}) \leq c_{2,\bm{u}}^E(S_{\bm u}^E \bm{u},\bm{u}) \label{eqn:equivalentHuE},\\
	&c_{1,p}^E(S_p^E p, p) \leq \left({(\ApEs)}^{-1} p, p\right) \leq c_{2,p}^E(S_p^E p , p) \label{eqn:equivalentHpE},
\end{eqnarray}
where the constants $c_{1,\bm{u}}^E$, $c_{1,p}^E$,
$c_{2,\bm{u}}^E$, and $c_{2,p}^E$ are independent of discretization and physical parameters.
In practice, $S_{\bm u}^E$ and
$S_{\bm w}$ can be defined similarly as in the full bubble case, and
$S_p^E$ can be defined through standard multigrid methods, as $\ApEs$ is  equivalent to a Poisson operator.

Since the preconditioners are derived directly from the well-posedness, they too are robust with respect to the physical and discretization parameters of the problem.
When applying the preconditioners to the bubble-eliminated formulation, though, a modest degradation in performance compared to the full bubble system is seen.
However, robustness with respect to the parameters remains. These properties are demonstrated in the numerical results section.

\subsection{Block Triangular Preconditioner}\label{sec:blocktri}
Next, we consider more general preconditioners, in particular,
block upper triangular and block lower triangular preconditioners for
the linear system, given by $\mathcal{A}$ or $\mathcal{A}^E$,
following the framework presented in
\cite{Adler2017,Klawonn1999,Loghin2004,Yicong2016,Starke1996} for FOV-equivalent preconditioners.
First, we define the notion of FOV equivalence as in \cite{Loghin2004}. Given a Hilbert space $X$ and its dual $X'$, a left preconditioner, $\mathcal{L}: X'\rightarrow X$, and a linear operator, $\mathcal{A}: X\rightarrow X'$, are \textit{FOV-equivalent} if, for any $x \in X$,
\begin{equation}
	\Sigma\leq\frac{\left(\mathcal{L}\mathcal{A}\bm{x},\bm{x}\right)_{\mathcal{N}^{-1}}}{\left(\bm{x},\bm{x}\right)_{\mathcal{N}^{-1}}},\
	\frac{\|\mathcal{L}\mathcal{A}\bm{x}\|_{\mathcal{N}^{-1}}}{\|\bm{x}\|_{\mathcal{N}^{-1}}}\leq\Upsilon.
\end{equation}
In general, $\mathcal{N}: X' \rightarrow X$ can be any SPD operator. Here, we choose $\mathcal{N}$ to be a SPD norm-equivalent preconditioner, and $\Sigma$ and $\Upsilon$ are positive constants\edit{, with $\Sigma \leq \Upsilon$}.  Using this definition, we have the following theorem on the convergence rate of preconditioned GMRES for solving $\mathcal{A} \bm{x} = \bm{f}$.
\begin{theorem}
	\cite{Eisenstat1983,Elman1982} If $\mathcal{A}$ and $\mathcal{L}$ are FOV-equivalent
	and $\bm{x}^m$ is the $m$-th iteration of the GMRES method preconditioned with $\mathcal{L}$, and $\bm{x}$ is the exact solution, then
	\begin{equation}
		\|\mathcal{L}\mathcal{A}(\bm{x}-\bm{x}^m)\|_{\mathcal{N}^{-1}}\leq\left(1-\frac{\Sigma^2}{\Upsilon^2}\right)^m \|\mathcal{L}\mathcal{A}(\bm{x}-\bm{x}^0)\|_{\mathcal{N}^{-1}}.
	\end{equation}
\end{theorem}
If the constants $\Sigma$ and $\Upsilon$ are independent of physical and discretization parameters, then $\mathcal{L}$ is a uniform left preconditioner
for GMRES.
\begin{remark}
Similar arguments apply to right preconditioners for GMRES, which are used in practice.
A right preconditioner, $\mathcal{R}: X'\rightarrow X$, and linear operator, $\mathcal{A}: X\rightarrow X'$ are \textit{FOV-equivalent} if, for any $x' \in X'$,
\begin{equation}
	\Sigma\leq\frac{\left(\mathcal{A}\mathcal{R}\bm{x}',\bm{x}'\right)_{\mathcal{N}}}{\left(\bm{x}',\bm{x}'\right)_{\mathcal{N}}},\
	\frac{\|\mathcal{A}\mathcal{R}\bm{x}'\|_{\mathcal{N}}}{\|\bm{x}'\|_{\mathcal{N}}}\leq\Upsilon.
\end{equation}
\end{remark}


\subsubsection{Full Bubble System}
For the three-field formulation, we first consider the block lower triangular preconditioner,
\begin{equation}\label{eqn:blk_low_tri}
	\mathcal{B}_L = \left(\begin{array}{ccc}
	A_{\bm u} & 0 & 0 \\
	-\alpha B_{\bm u} & \left(\frac{\alpha^2}{\zeta^2} + \frac{1}{M}\right) M_p & 0 \\
	0 & \tau B_{\bm w}^T & \tau M_{\bm w} + \tau^2 \left(\frac{\alpha^2}{\zeta^2}+\frac{1}{M}\right)^{-1} A_{\bm w}
	\end{array}\right)^{-1},
\end{equation}
and the inexact block lower triangular preconditioner,
\begin{equation}\label{eqn:blk_low_tri_M}
	\widehat{\mathcal{B}_L} = \left(\begin{array}{ccc}
	S_{\bm u}^{-1} & 0 & 0 \\
	-\alpha B_{\bm u} & S_p^{-1} & 0 \\
	0 & \tau B_{\bm w}^T & S_{\bm w}^{-1}
	\end{array}\right)^{-1}.
\end{equation}

\begin{theorem} \label{thm:BL}
	Assuming a shape regular mesh and the discretization described above,
	there exist constants $\Sigma$ and $\Upsilon$, independent of discretization and physical parameters,
	such that, for any $\bm{x}=(\bm{u},p,\bm{w})^T\neq\bm{0}$,
	\[
	\Sigma\leq\frac{\left(\mathcal{B}_L\mathcal{A}\bm{x},\bm{x}\right)_{(\mathcal{B}_D)^{-1}}}{\left(\bm{x},\bm{x}\right)_{(\mathcal{B}_D)^-1}},\
	\frac{\|\mathcal{B}_L\mathcal{A}\bm{x}\|_{(\mathcal{B}_D)^{-1}}}{\|\bm{x}\|_{(\mathcal{B}_D)^{-1}}}\leq\Upsilon.
	\]
\end{theorem}
\begin{theorem} \label{thm:BL_inexact}
	Assuming the spectral equivalence relations
	(\ref{eqn:equivalentHu}) and (\ref{eqn:equivalentHw}) hold, $\| I-S_u A_{\bm u} \|_{A_{\bm u}} \leq \rho \leq 0.2228$,
	and $S_p = \left(\frac{\alpha^2}{\zeta^2} + \frac{1}{M}\right)^{-1} M_p^{-1}$,
	then there exists constants $\Sigma$ and $\Upsilon$, independent of discretization and physical parameters,
	such that, for any $\bm{x}=(\bm{u},p,\bm{w})^T\neq\bm{0}$,
	\[
	\Sigma\leq\frac{\left(\widehat{\mathcal{B}_L}\mathcal{A} \bm{x},\bm{x}\right)_{(\widehat{\mathcal{B}_D})^{-1}}}
		{\left(\bm{x},\bm{x}\right)_{(\widehat{\mathcal{B}_D})^-1}},\
	\frac{\|\widehat{\mathcal{B}_L}\mathcal{A} \bm{x}\|_{(\widehat{\mathcal{B}_D})^{-1}}}{\|\bm{x}\|_{(\widehat{\mathcal{B}_D})^{-1}}}\leq\Upsilon.
	\]
\end{theorem}
The proofs of the above two theorems turn out to be a special case of the proofs for the bubble-eliminated system (shown below), and thus are omitted here.

Similar arguments can also be applied to block upper triangular
preconditioners.  We consider the following for $\mathcal{A}$ in \eqref{block_form},
\begin{equation}\label{eqn:blk_up_triF}
	\mathcal{B}_U = \left(\begin{array}{ccc}
	A_{\bm u} & \alpha B_{\bm u}^T & 0 \\
	0 & \left(\frac{\alpha^2}{\zeta^2} + \frac{1}{M}\right) M_p & -\tau B_{\bm w}  \\
	0 & 0 & \tau M_{\bm w} + \tau^2 \left(\frac{\alpha^2}{\zeta^2}+\frac{1}{M}\right)^{-1} A_{\bm w}
	\end{array}\right)^{-1},
\end{equation}
and the corresponding inexact preconditioner,
\begin{equation}\label{eqn:blk_up_tri_MF}
	\widehat{\mathcal{B}_U} = \left(\begin{array}{ccc}
	S_{\bm u}^{-1} & \alpha B_{\bm u}^T & 0 \\
	0 & S_{p}^{-1} & -\tau B_{\bm w} \\
	0 & 0 & S_{\bm w}^{-1}
	\end{array}\right)^{-1}.
\end{equation}
Parameter robustness for the block upper
triangular preconditioners is summarized in the following theorems.  Again, as these results are special cases of those in the following section, we only state the results here.
\begin{theorem}
	Assuming a shape regular mesh and the discretization described above,
	there exist constants $\Sigma$ and $\Upsilon$, independent of discretization and physical parameters,
	such that, for any $\bm{x}' = \mathcal{B}_U^{-1}\bm{x}$ with $\bm{x}=(\bm{u},p,\bm{w})^T\neq\bm{0}$,
	\[
	\Sigma\leq\frac{\left(\mathcal{A}\mathcal{B}_U\bm{x}',\bm{x}'\right)_{(\mathcal{B}_D)}}{\left(\bm{x}',\bm{x}'\right)_{(\mathcal{B}_D)}},\
	\frac{\|\mathcal{A}\mathcal{B}_U\bm{x}'\|_{(\mathcal{B}_D)}}{\|\bm{x}'\|_{(\mathcal{B}_D)}}\leq\Upsilon.
	\]
\end{theorem}

\begin{theorem}
	Assuming the spectral equivalence relations
	(\ref{eqn:equivalentHu}) and (\ref{eqn:equivalentHw}) hold,
	$\| I-A_{\bm u}S_{\bm u}\|_{A_{\bm u}}\leq\rho\leq 0.2228$,
	and $S_p = \left(\frac{\alpha^2}{\zeta^2} + \frac{1}{M}\right)^{-1} M_p^{-1}$, then
	there exists constants $\Sigma$ and $\Upsilon$, independent of discretization and physical parameters,
	such that, for any $\bm{x}' = \mathcal{B}_U^{-1}\bm{x}$ with $\bm{x}=(\bm{u},p,\bm{w})^T\neq\bm{0}$,
	\[
	\Sigma\leq\frac{\left(\mathcal{A}\widehat{\mathcal{B}_U}\bm{x}',\bm{x}'\right)_{(\widehat{\mathcal{B}_D})}}
		{\left(\bm{x}',\bm{x}'\right)_{(\widehat{\mathcal{B}_D})}},\
	\frac{\|\mathcal{A}\widehat{\mathcal{B}_U}\bm{x}'\|_{(\widehat{\mathcal{B}_D})}}
		{\|\bm{x}'\|_{(\widehat{\mathcal{B}_D})}}\leq\Upsilon.
	\]
\end{theorem}

\subsubsection{Bubble-Eliminated System}
For the three-field formulation, we consider the block lower triangular preconditioner,
\begin{equation}\label{eqn:blk_low_triE}
	\mathcal{B}_L^E = \left(\begin{array}{ccc}
	\AuE & 0 & 0 \\
	-\alpha \BuE & \ApEs & 0 \\
	0 & \tau B_{\bm w}^T & \tau M_{\bm w} + \tau^2 \left(\frac{\alpha^2}{\zeta^2}+\frac{1}{M}\right)^{-1} A_{\bm w}
	\end{array}\right)^{-1},
\end{equation}
and the inexact block lower triangular preconditioner,
\begin{equation}\label{eqn:blk_low_tri_ME}
	\widehat{\mathcal{B}_L^E} = \left(\begin{array}{ccc}
	{S_{\bm u}^E}^{-1} & 0 & 0 \\
	-\alpha \BuE & {S_p^E}^{-1} & 0 \\
	0 & \tau B_{\bm w}^T & S_{\bm w}^{-1}
	\end{array}\right)^{-1},
\end{equation}
where 
$\BuE = B_l- B_b D_{bb}^{-1} A_{bl}$.
For notational convenience, we define\\ $\ApE = \frac{1}{M} M_p+\alpha^2 B_b D_{bb}^{-1} B_{b}^T$ as the pressure block in the bubble-eliminated system.

\edit{\begin{lemma} \label{lem:stokesP}
	If the pair of finite-element spaces $\bm{V}_h \times Q_h$ is Stokes-stable, i.e., they satisfy the following inf-sup condition~\cite{GR1986},
	\begin{equation}\label{eqn:stokes-inf-sup}
		\sup_{\bm{v}\in\bm{V}_h} \frac{(\ddiv\bm{v},p)}{\|\bm{v}\|_1} \geq \gamma_B^0 \|p\|, \quad \forall \ p\in Q_h,
	\end{equation}
	then, in matrix form, we have,
	\begin{equation}\label{eqn:Ainv_to_L2}
		\|B_{\bm{u}}^Tp\|_{A_{\bm{u}}^{-1}} \geq \frac{\gamma_B}{\zeta}\|p\|_{M_p}, \quad \forall \ p\in Q_h,
	\end{equation}
	with $\gamma_B = \gamma_B^0/\sqrt{d}$.
	Furthermore, from \eqref{eqn:ADequiv}, we have,
	\begin{equation}\label{eqn:ADinv_to_L2}
		\|B_{\bm{u}}^Tp\|_{(A^D_{\bm{u}})^{-1}} \geq \frac{1}{\eta} \frac{\gamma_B}{\zeta}\|p\|_{M_p}, \quad \forall \ p\in Q_h.
	\end{equation}
\end{lemma}
\begin{proof}
Here, we use $\bm{v}$ to denote both the finite-element function and its vector representation.  Since $\bm{V}_h \times Q_h$ is Stokes stable, it satisfies the inf-sup condition in \eqref{eqn:stokes-inf-sup},
	where $\gamma^0_B > 0$ is a constant that does not depend on mesh size.
	Using the fact that $a(\bm{u},\bm{u}) \leq (2\mu +
        d\lambda)(\epsilon(\bm{u}),\epsilon(\bm{u}))$, we have $\|\bm{v}\|_{A_{\bm{u}}} \leq \sqrt{d}\zeta\|\bm{v}\|_1$. Then, for any $p\in Q_h$,
	\begin{equation}\label{eqn:stokesA-inf-sup}
		\sup_{\bm{v}\in\bm{V_h}} \frac{(B_{\bm{u}}\bm{v},p)}{\|\bm{v}\|_{A_{\bm{u}}}} \geq \sup_{\bm{v}\in\bm{V_h}} \frac{(\ddiv \bm{v},p)}{\sqrt{d}\zeta\|\bm{v}\|_1}
		\geq \frac{\gamma_B^0}{\sqrt{d}\zeta}\|p\|_{M_p} =: \frac{\gamma_B}{\zeta}\|p\|_{M_p}.
	\end{equation}
	Using (\ref{eqn:stokesA-inf-sup}) and
	\begin{equation*}
		\|B_{\bm u}^Tp\|_{A_{\bm{u}}^{-1}} = \sup_{\bm{v}\in\bm{V_h}} \frac{(\bm{v},B_{\bm u}^Tp)}{\|\bm{v}\|_{A_{\bm{u}}}} =
		\sup_{\bm{v}\in\bm{V_h}} \frac{(B_{\bm{u}}\bm{v},p)}{\|\bm{v}\|_{A_{\bm{u}}}},
	\end{equation*}
	(\ref{eqn:Ainv_to_L2}) is obtained. Equation~\eqref{eqn:ADinv_to_L2} follows from \eqref{eqn:ADequiv}, the spectral equivalence of $A_{\bm{u}}$ and $A_{\bm{u}}^D$.
\end{proof}}

In order to prove that (\ref{eqn:blk_low_triE}) and (\ref{eqn:blk_low_tri_ME}) satisfy the requirements to be FOV-equivalent preconditioners for the $\mathcal{A}^E$ system we need the \edit{following} relation for the bubble-eliminated system,
\begin{equation}\label{eqn:Bp_to_p_elim}
 \|(\BuE)^T p\|_{{(\AuE)}^{-1}}^2 \geq \frac{\gamma_B^2}{\eta^2 \zeta^2} \| p \|_{M_p}^2 - (D_{bb}^{-1} B_b^T p, B_b^T p).
 \end{equation}
This is established using Lemma~\ref{lem:stokesP}, the first two by two blocks of Equation~(\ref{eqn:LSL}), and a direct computation.
With this result, we now show that (\ref{eqn:blk_low_triE}) satisfies the requirements to be an FOV-equivalent preconditioner for $\mathcal{A}^E$.

\begin{theorem} \label{thm:BL_elim}
	Assuming a shape regular mesh and the discretization described above,
	there exists constants $\Sigma$ and $\Upsilon$, independent of discretization or physical parameters,
	such that, for any $\bm{x}=(\bm{u},p,\bm{w})^T\neq\bm{0}$,
	\[
	\Sigma\leq\frac{\left(\mathcal{B}_L^E\mathcal{A}^E \bm{x},\bm{x}\right)_{(\mathcal{B}_D^E)^{-1}}}{\left(\bm{x},\bm{x}\right)_{(\mathcal{B}_D^E)^-1}},\
	\frac{\|\mathcal{B}_L^E\mathcal{A}^E \bm{x}\|_{(\mathcal{B}_D^E)^{-1}}}{\|\bm{x}\|_{(\mathcal{B}_D^E)^{-1}}}\leq\Upsilon.
	\]
\end{theorem}
\begin{proof}
    By direct computation and the Cauchy-Schwarz inequality, 
\begin{align*}
(\mathcal{B}_L^E\mathcal{A}^E \bm{x},\bm{x})_{(\mathcal{B}_D^E)^{-1}}
&= \|u\|_{\AuE}^2 + \alpha((\BuE)^Tp,u) + \alpha^2 \|(\BuE)^T p\|^2_{(\AuE)^{-1}} + \|p\|_{(\ApE)}^2 \\
	&\quad - \tau \alpha^2 ({(\ApEs)}^{-1} (\BuE) (\AuE)^{-1} (\BuE)^T p, B_{\bm w} \bm{w})\\
	&\quad - \tau({(\ApEs)}^{-1} (\ApE) p, B_{\bm w} \bm{w}) + \tau^2 \|B_{\bm w} \bm{w}\|_{{(\ApEs)}^{-1}}^2 + \tau \|\bm{w}\|_{M_{\bm w}}^2\\
&\geq \|u\|_{\AuE}^2 - \alpha\|(\BuE)^Tp\|_{(\AuE)^{-1}} \|u\|_{\AuE} + \alpha^2 \|(\BuE)^T p\|^2_{(\AuE)^{-1}} \\
    & \quad  + \|p\|_{(\ApE)}^2 - \tau \alpha^2 \| (\BuE) (\AuE)^{-1} (\BuE)^T p\|_{{(\ApEs)}^{-1}} \|B_{\bm w} \bm{w}\|_{{(\ApEs)}^{-1}} \\
	&\quad - \tau\|(\ApE) p\|_{{(\ApEs)}^{-1}} \|B_{\bm w} \bm{w}\|_{{(\ApEs)}^{-1}} \\
	&\quad  + \tau^2 \|B_{\bm w} \bm{w}\|_{{(\ApEs)}^{-1}}^2 + \tau \|\bm{w}\|_{M_{\bm w}}^2.
	\end{align*}
	By the definitions of matrices $A_p^{E_*}$, $M_p$, and $A_p^E$, we have
	\begin{align}
	\| q \|_{(\ApEs)^{-1}}^2 &\leq \left(\frac{\alpha^2}{\zeta^2}+\frac{1}{M}\right)^{-1} \| q \|_{M_p^{-1}}^2, \label{eqn:Aps_leq_Mp}\\
	\| q \|_{{(\ApEs)}^{-1}}^2 &\leq \| q \|_{(\ApE)^{-1}}^2. \label{eqn:Aps_leq_Ap}
	\end{align}
Then, using (\ref{eqn:Aps_leq_Mp}) on the $\| (\BuE) (\AuE)^{-1} (\BuE)^T p\|_{{(\ApEs)}^{-1}}$ term and (\ref{eqn:Aps_leq_Ap}) on the $\|(\ApE) p\|_{{(\ApEs)}^{-1}}$ term, we obtain,
\begin{align*}
(\mathcal{B}_L^E\mathcal{A}^E \bm{x},\bm{x})_{(\mathcal{B}_D^E)^{-1}}
&\geq \|u\|_{\AuE}^2 - \alpha\|(\BuE)^Tp\|_{(\AuE)^{-1}} \|u\|_{\AuE} \\
	&\quad + \alpha^2 \|(\BuE)^T p\|^2_{(\AuE)^{-1}}  + \|p\|_{(\ApE)}^2 \\
	&\quad- \tau \alpha^2 \left(\frac{\alpha^2}{\zeta^2}+\frac{1}{M}\right)^{-\frac{1}{2}} \|\BuE(\AuE)^{-1} (\BuE)^T p\|_{M_p^{-1}} \|B_{\bm w} \bm{w}\|_{{(\ApEs)}^{-1}} \\
	&\quad - \tau \| p\|_{(\ApE)} \|B_{\bm w} \bm{w}\|_{{(\ApEs)}^{-1}} + \tau^2 \|B_{\bm w} \bm{w}\|_{{(\ApEs)}^{-1}}^2 + \tau \|\bm{w}\|_{M_{\bm w}}^2.
	\end{align*}
	Observing that, for $d=2,3$, $a(\bm{v},\bm{v}) \leq (2\mu + d \lambda)(\epsilon(\bm{v}),\epsilon(\bm{v}))$ for any $\bm{v}$,
	and a direct computation of the elimination of the bubble, we have,
	\begin{equation}\label{eqn:Bu_leq_Au}
	\zeta^2 \| \BuE \bm{v} \|_{{M_p}^{-1}}^2 \leq \| \bm{v} \|_{\AuE}^2.
	\end{equation}
Applying (\ref{eqn:Bu_leq_Au}) to the $\|\BuE(\AuE)^{-1} (\BuE)^T p\|_{M_p^{-1}}$ term with $\bm{v}= (\AuE)^{-1} (\BuE)^T p$ gives,
\begin{align*}
(\mathcal{B}_L^E\mathcal{A}^E \bm{x},\bm{x})_{(\mathcal{B}_D^E)^{-1}}
&\geq \|u\|_{\AuE}^2 - \alpha\|(\BuE)^Tp\|_{(\AuE)^{-1}} \|u\|_{\AuE} \\
	& \quad + \alpha^2 \|(\BuE)^T p\|^2_{(\AuE)^{-1}} + \|p\|_{(\ApE)}^2 \\
	&\quad - \tau \alpha \frac{\alpha}{\zeta} \left(\frac{\alpha^2}{\zeta^2}+\frac{1}{M}\right)^{-\frac{1}{2}}
		\| (\BuE)^T p\|_{{(\AuE)}^{-1}} \|B_{\bm w} \bm{w}\|_{{(\ApEs)}^{-1}} \\
	&\quad - \tau \| p\|_{(\ApE)} \|B_{\bm w} \bm{w}\|_{{(\ApEs)}^{-1}} + \tau^2 \|B_{\bm w} \bm{w}\|_{{(\ApEs)}^{-1}}^2 + \tau \|\bm{w}\|_{M_{\bm w}}^2\\
&\geq \|u\|_{\AuE}^2 - \alpha\|(\BuE)^Tp\|_{(\AuE)^{-1}} \|u\|_{\AuE} + \alpha^2 \|(\BuE)^T p\|^2_{(\AuE)^{-1}} \\
	& \quad + \|p\|_{(\ApE)}^2 - \tau \alpha \| (\BuE)^T p\|_{({\AuE})^{-1}} \|B_{\bm w} \bm{w}\|_{{(\ApEs)}^{-1}} \\
	&\quad - \tau \| p\|_{(\ApE)} \|B_{\bm w} \bm{w}\|_{{(\ApEs)}^{-1}} + \tau^2 \|B_{\bm w} \bm{w}\|_{{(\ApEs)}^{-1}}^2 + \tau \|\bm{w}\|_{M_{\bm w}}^2,
\end{align*}
where we use the fact $\frac{\alpha}{\zeta} \left(\frac{\alpha^2}{\zeta^2}+\frac{1}{M}\right)^{-\frac{1}{2}} < 1$. Rewriting the right hand side,
\begin{align*}
&\ \quad (\mathcal{B}_L^E\mathcal{A}^E \bm{x},\bm{x})_{(\mathcal{B}_D^E)^{-1}}\geq \\
&
  \left(\begin{array}{c}
  \|\bm{u}\|_{\AuE} 			\\
  \alpha \|(\BuE)^T p\|_{(\AuE)^{-1}}	\\
  \|p\|_{(\ApE)}	\\
  \tau \|B_{\bm w} \bm{w}\|_{{(\ApEs)}^{-1}}	\\
  \sqrt{\tau}\|\bm{w}\|_{M_{\bm w}}
  \end{array}\right)^T
  \left(\begin{array}{ccccc}
  1&-\frac{1}{2}&0&0&0	\\
  -\frac{1}{2}&1&0&-\frac{1}{2}&0	\\
  0&0&1&-\frac{1}{2}&0	\\
  0&-\frac{1}{2}&-\frac{1}{2}&1&0	\\
  0&0&0&0&1
  \end{array}\right)
  \left(\begin{array}{c}
  \|\bm{u}\|_{\AuE} 			\\
  \alpha \|(\BuE)^T p\|_{(\AuE)^{-1}}	\\
  \|p\|_{(\ApE)}	\\
  \tau \|B_{\bm w} \bm{w}\|_{{(\ApEs)}^{-1}}	\\
  \sqrt{\tau}\|\bm{w}\|_{M_{\bm w}}
  \end{array}\right).
\end{align*}
The above matrix is SPD, meaning that there is a $\sigma > 0$ such that
\begin{align*}
(\mathcal{B}_L^E\mathcal{A}^E \bm{x},\bm{x})_{(\mathcal{B}_D^E)^{-1}}
&\geq \sigma\left(\|\bm{u}\|_{\AuE}^2 + \alpha^2 \|(\BuE)^T p\|_{(\AuE)^{-1}}^2 + \|p\|_{(\ApE)}^2 \right.  \\
		& \left. \quad \quad  \  + \tau^2 \|B_{\bm w} \bm{w}\|_{{(\ApEs)}^{-1}}^2 + \tau\|\bm{w}\|_{M_{\bm w}}^2 \right)\\
&\geq \sigma\left(\|\bm{u}\|_{\AuE}^2 + \frac{\alpha^2}{2} \|(\BuE)^T p\|_{(\AuE)^{-1}}^2  + \|p\|_{(\ApE)}^2  \right. \\
		& \left. \quad\quad \ + \tau^2 \|B_{\bm w} \bm{w}\|_{{(\ApEs)}^{-1}}^2 + \tau\|\bm{w}\|_{M_{\bm w}}^2 \right).
\end{align*}
Using (\ref{eqn:Bp_to_p_elim}), and the definition of $\| p \|_{\ApE}$, we get
\begin{align*}
(\mathcal{B}_L^E\mathcal{A}^E \bm{x},\bm{x})_{(\mathcal{B}_D^E)^{-1}}
&\geq \sigma\left(\|\bm{u}\|_{\AuE}^2 + \frac{\gamma_B^2}{2\eta^2} \frac{\alpha^2}{\zeta^2} \|p\|_{M_p}^2
		- \frac{\alpha^2}{2}(D_{bb}^{-1} B_b^T p, B_b^T p) + \|p\|_{(\ApE)}^2 \right .\\
&\qquad\quad	\left .	+ ~\tau^2 \|B_{\bm w} \bm{w}\|_{{(\ApEs)}^{-1}}^2 + \tau\|\bm{w}\|_{M_{\bm w}}^2 \right)\\
&= \sigma\left(\|\bm{u}\|_{\AuE}^2 + \frac{\gamma_B^2}{2\eta^2} \frac{\alpha^2}{\zeta^2} \|p\|_{M_p}^2 + \frac{1}{M}\|p\|_{M_p}^2
		+ \frac{\alpha^2}{2}(D_{bb}^{-1} B_b^T p, B_b^T p)\right .\\
		&\qquad\quad \left . + ~\tau^2 \|B_{\bm w} \bm{w}\|_{{(\ApEs)}^{-1}}^2 + \tau\|\bm{w}\|_{M_{\bm w}}^2 \right)\\
&\geq \sigma\left(\|\bm{u}\|_{\AuE}^2 + \min\{\frac{\gamma_B^2}{\eta^2},1\}\frac{1}{2}\|p\|_{(\ApEs)}^2
		+ \tau^2 \|B_{\bm w} \bm{w}\|_{{(\ApEs)}^{-1}}^2\right .\\
		&\qquad\quad \left . + ~\tau\|\bm{w}\|_{M_{\bm w}}^2 \right)\\
&\geq \sigma\left(\|\bm{u}\|_{\AuE}^2 + \min\{\frac{\gamma_B^2}{\eta^2},1\}\frac{1}{2}\|p\|_{(\ApEs)}^2
		 \right .\\
		&\qquad\quad \left . + ~\tau^2\left( \frac{\alpha^2}{\zeta^2} + \frac{1}{M}\right)^{-1} \| B_{\bm{w}} \bm{w}\|_{M_p^{-1}}^2 + \tau\|\bm{w}\|_{M_{\bm w}}^2 \right)\\
&\geq \Sigma(\bm{x},\bm{x})_{(\mathcal{B}_D^E)^{-1}},
	\end{align*}
where $\Sigma = \sigma\frac{1}{2}\min\{1,\frac{\gamma_B^2}{\eta^2}\}$. This provides the lower bound for the bubble-eliminated case.
The upper bound follows from the continuity of each term and the Cauchy-Schwarz inequality.
\end{proof}

Next, we prove that (\ref{eqn:blk_low_tri_ME}) satisfies the requirements to be an FOV-equivalent preconditioner for the $\mathcal{A}^E$ system when the inexact diagonal blocks are solved to sufficient accuracy.
\begin{theorem} \label{thm:BL_elim_inexact} 
	Assuming the spectral equivalence relations
	(\ref{eqn:equivalentHw}), (\ref{eqn:equivalentHuE}), and (\ref{eqn:equivalentHpE}) hold,
	$\| I-S_u^E \AuE \|_{\AuE} \leq \rho$ and $\| I-S_p^E (\ApEs) \|_{(\ApEs)} \leq \beta$, with $\rho>0$ and $\beta>0$ sufficiently small, then
	there exist constants $\Sigma$ and $\Upsilon$, independent of discretization and physical parameters,
	such that, for any $\bm{x}=(\bm{u},p,\bm{w})^T\neq\bm{0}$,
	\[
	\Sigma\leq\frac{\left(\widehat{\mathcal{B}_L^E}\mathcal{A}^E \bm{x},\bm{x}\right)_{(\widehat{\mathcal{B}_D^E})^{-1}}}
		{\left(\bm{x},\bm{x}\right)_{(\widehat{\mathcal{B}_D^E})^-1}},\
	\frac{\|\widehat{\mathcal{B}_L^E}\mathcal{A}^E \bm{x}\|_{(\widehat{\mathcal{B}_D^E})^{-1}}}{\|\bm{x}\|_{(\widehat{\mathcal{B}_D^E})^{-1}}}\leq\Upsilon.
	\]
\end{theorem}
\begin{proof}
Assume that $\| I - S_u^E \AuE \|_{\AuE} \leq \rho$ and that $\| I - S_p^E (\ApEs) \|_{(\ApEs)} \leq \beta$.
By direct computation,
\begin{align*}
(\widehat{\mathcal{B}_L^E}\mathcal{A}^E \bm{x},\bm{x})_{\widehat{\mathcal{B}_D^E}^{-1}}
&= \|\bm{u}\|_{\AuE}^2 + \alpha((\BuE)^Tp,S_u^E \AuE \bm{u}) + \alpha^2 \|(\BuE)^T p\|^2_{S_u^E} + \|p\|_{(\ApE)}^2 \\
	&\quad + \tau \alpha (S_p^E (\BuE) (I-S_u^E \AuE) \bm{u}, B_{\bm w} \bm{w})\\
	&\quad - \tau \alpha^2 (S_p^E (\BuE) S_u^E (\BuE)^T p, B_{\bm w} \bm{w})\\
	&\quad - \tau(S_p^E (\ApE) p, B_{\bm w} \bm{w}) + \tau^2 \|B_{\bm w} \bm{w}\|_{S_p^E}^2 + \tau \|\bm{w}\|_{M_{\bm w}}^2\\
&\geq \|\bm{u}\|_{\AuE}^2 - \alpha\|(\BuE)^Tp\|_{S_u^E} \|\AuE \bm{u}\|_{S_u^E} + \alpha^2 \|(\BuE)^T p\|^2_{S_u^E} + \|p\|_{(\ApE)}^2\\
	&\quad -\tau \alpha \| (\BuE) (I-S_u^E \AuE) \bm{u} \|_{S_p^E} \| B_{\bm w} \bm{w} \|_{S_p^E}\\
	&\quad - \tau \alpha^2 \| (\BuE) S_u^E (\BuE)^T p \|_{S_p^E}  \| B_{\bm w} \bm{w} \|_{S_p^E}\\
	&\quad - \tau \| (\ApE) p\|_{S_p^E} \|B_{\bm w} \bm{w}\|_{S_p^E} + \tau^2 \|B_{\bm w} \bm{w}\|_{S_p^E}^2 + \tau \|\bm{w}\|_{M_{\bm w}}^2.
\end{align*}
Using $\| I - S_u^E \AuE \|_{\AuE} \leq \rho$ and $\| I - S_p^E \ApEs \|_{\ApEs} \leq \beta$, on  $\|\AuE u\|_{S_u^E}$,\\ $\| \BuE (I-S_u^E \AuE) u \|_{S_p^E}$, $\| \BuE S_u^E (\BuE)^T p \|_{S_p^E}$, and $\| \ApE p\|_{S_p^E}$ allows us to change norms and apply (\ref{eqn:Bu_leq_Au}), (\ref{eqn:Aps_leq_Mp}), and (\ref{eqn:Aps_leq_Ap}) to these terms in the same way as in the previous proof.  Thus,
\begin{align*}
(\widehat{\mathcal{B}_L^E}\mathcal{A}^E \bm{x},\bm{x})_{\widehat{\mathcal{B}_D^E}^{-1}}
&\geq \|u\|_{\AuE}^2 - \alpha(1+\rho)\|(\BuE)^Tp\|_{S_u^E} \|\bm{u}\|_{\AuE} + \alpha^2 \|(\BuE)^T p\|^2_{S_u^E} \\
	&\quad + \|p\|_{(\ApE)}^2 -\tau \frac{\alpha}{\zeta} (1+\beta) \rho \left(\frac{\alpha^2}{\zeta^2}+\frac{1}{M}\right)^{-\frac{1}{2}} \| \bm{u} \|_{\AuE} \| B_{\bm w} \bm{w} \|_{S_p^E}\\
	&\quad - \tau \alpha^2 \frac{1}{\zeta}(1+\beta)(1+\rho) \left(\frac{\alpha^2}{\zeta^2}+\frac{1}{M}\right)^{-\frac{1}{2}} \| (\BuE)^T p \|_{S_u^E}  \| B_{\bm w} \bm{w} \|_{S_p^E}\\
	&\quad - \tau (1+\beta) \| p\|_{(\ApE)} \|B_{\bm w} \bm{w}\|_{S_p^E} + \tau^2 \|B_{\bm w} \bm{w}\|_{S_p^E}^2 + \tau \|\bm{w}\|_{M_{\bm w}}^2\\
&\geq \|\bm{u}\|_{\AuE}^2 - \alpha(1+\rho)\|(\BuE)^Tp\|_{S_u^E} \|\bm{u}\|_{\AuE} + \alpha^2 \|(\BuE)^T p\|^2_{S_u^E}\\
	&\quad + \|p\|_{(\ApE)}^2 -\tau (1+\beta) \rho \| \bm{u} \|_{\AuE} \| B_{\bm w} \bm{w} \|_{S_p^E}\\
	&\quad - \tau \alpha (1+\beta)(1+\rho) \| (\BuE)^T p \|_{S_u^E}  \| B_{\bm w} \bm{w} \|_{S_p^E}\\
	&\quad - \tau (1+\beta) \| p\|_{(\ApE)} \|B_{\bm w} w\|_{S_p^E} + \tau^2 \|B_{\bm w} \bm{w}\|_{S_p^E}^2 + \tau \|\bm{w}\|_{M_{\bm w}}^2.
\end{align*}
Then, rewriting the right hand side,
\begin{align*}
&(\widehat{\mathcal{B}_L^E}\mathcal{A}^E \bm{x},\bm{x})_{\widehat{\mathcal{B}_D^E}^{-1}} \geq
	\left(\begin{array}{c}
	\|\bm{u}\|_{\AuE} 			\\
	\alpha \|(\BuE)^T p\|_{S_u^E}	\\
	\|p\|_{(\ApE)}	\\
	\tau \|B_{\bm w} \bm{w}\|_{S_p^E}	\\
	\sqrt{\tau}\|\bm{w}\|_{M_{\bm w}}
	\end{array}\right)^T
	\mathcal{Q}
	\left(\begin{array}{c}
	\|\bm{u}\|_{\AuE} 			\\
	\alpha \|(\BuE)^T p\|_{S_u^E}	\\
	\|p\|_{(\ApE)}	\\
	\tau \|B_{\bm w} \bm{w}\|_{S_p^E}	\\
	\sqrt{\tau}\|\bm{w}\|_{M_{\bm w}}
	\end{array}\right),
\end{align*}
where
\begin{align*}
\mathcal{Q} =& \\
&\left(\begin{array}{ccccc}
1				&-\frac{1}{2}(1+\rho)			&0					&-\frac{1}{2}\rho(1+\beta)		&0	\\
-\frac{1}{2}(1+\rho)	&1						&0					&-\frac{1}{2}(1+\rho)(1+\beta)&0	\\
0				&0						&1					&-\frac{1}{2}(1+\beta)		&0	\\
-\frac{1}{2}\rho(1+\beta)	&-\frac{1}{2}(1+\rho)(1+\beta)&-\frac{1}{2}(1+\beta)	&1					&0	\\
0				&0						&0					&0					&1
\end{array}\right).
\end{align*}
If $\beta$ and $\rho$ are sufficiently small, 
then the above matrix is SPD, and there is a $\sigma>0$ such that
\begin{align*}
(\widehat{\mathcal{B}_L^E}\mathcal{A}^E\bm{x},\bm{x})_{\widehat{\mathcal{B}_D^E}^{-1}}
&\geq \sigma\left(\|\bm{u}\|_{\AuE}^2 + \alpha^2 \|(\BuE)^T p\|_{S_u^E}^2 + \|p\|_{(\ApE)}^2
		+ \tau^2 \|B_{\bm w} \bm{w}\|_{S_p^E}^2\right .\\
		&\qquad\quad \left .  + ~\tau\|\bm{w}\|_{M_{\bm w}}^2 \right)\\
&\geq \sigma\left( (1-\rho) \|\bm{u}\|_{(S_u^E)^{-1}}^2 + (1-\rho)\frac{\gamma_B^2}{2\eta^2}\frac{\alpha^2}{\zeta^2} \|p\|_{M_p}^2
		 \right.\\
		&\qquad\quad\left. - ~(1-\rho)\frac{\alpha^2}{2}(D_{bb}^{-1} B_b^T p, B_b^T p) + \|p\|_{(\ApE)}^2 + \tau^2 \|B_{\bm w} \bm{w}\|_{S_p^E}^2 \right .\\
		&\qquad\quad\left. + ~\tau\|\bm{w}\|_{M_{\bm w}}^2 \right)\\
&\geq \sigma\left( (1-\rho) \|\bm{u}\|_{(S_u^E)^{-1}}^2 + (1-\rho)\frac{\gamma_B^2}{2\eta^2}\frac{\alpha^2}{\zeta^2} \|p\|_{M_p}^2
		 \right.\\
		&\qquad\quad\left. + ~(1+\rho)\frac{\alpha^2}{2}(D_{bb}^{-1} B_b^T p, B_b^T p) + \frac{1}{M}\|p\|_{M_p}^2 + \tau^2 \|B_{\bm w} \bm{w}\|_{S_p}^2\right .\\
		&\qquad\quad \left. + ~\tau\|\bm{w}\|_{M_{\bm w}}^2 \right)\\
&\geq \sigma\left( (1-\rho) \|\bm{u}\|_{(S_u^E)^{-1}}^2 + \frac{(1-\rho)(1-\beta)}{2}\min(1,\frac{\gamma_B^2}{\eta^2})\|p\|_{(S_p^E)^{-1}}^2 \right.\\
		&\quad\quad\left. + \tau^2 (1-\beta) \left(\frac{\alpha^2}{\zeta^2}+\frac{1}{M}\right)^{-1} \|B_{\bm{w}} \bm{w}\|_{M_p^{-1}}^2 + \tau\|\bm{w}\|_{M_{\bm w}}^2 \right)\\
&\geq \Sigma(\bm{x},\bm{x})_{(\widehat{\mathcal{B}_D^E})^{-1}},
\end{align*}
where $\Sigma = \sigma\frac{(1-\rho)(1-\beta)}{2}\min\{1,\frac{\gamma_B^2}{\eta^2}\}$. This provides the lower bound.
The upper bound follows from the continuity of each term and the Cauchy-Schwarz inequality.
\end{proof}
\begin{remark}
Values for $\beta$ and $\rho$ that are sufficiently small can be calculated numerically.
For example, if $0 < \beta  = \rho < 0.1291$, then the above matrix is SPD.
\end{remark}

Similar arguments can also be applied to block upper triangular
preconditioners.  We consider the following upper preconditioner for $\mathcal{A}^E$ in \eqref{block_form_elim},
\begin{equation}\label{eqn:blk_up_tri}
	\mathcal{B}_U^E = \left(\begin{array}{ccc}
	\AuE & \alpha (\BuE)^T & 0 \\
	0 & \ApEs & -\tau B_{\bm w}  \\
	0 & 0 & \tau M_{\bm w} + \tau^2 \left(\frac{\alpha^2}{\zeta^2}+\frac{1}{M}\right)^{-1} A_{\bm w}
	\end{array}\right)^{-1}
\end{equation}
where, again,
$\AuE =  A_{ll}-A_{bl}^T D_{bb}^{-1}A_{bl}$,
$\ApEs = \left(\frac{\alpha^2}{\zeta^2} + \frac{1}{M}\right) M_p+\alpha^2 B_b D_{bb}^{-1} B_{b}^T$, and\\
$\BuE = B_l- B_b D_{bb}^{-1} A_{bl}$ when preconditioning the bubble-eliminated case.
The corresponding inexact preconditioner is given by:
\begin{equation}\label{eqn:blk_up_tri_M}
	\widehat{\mathcal{B}_U^E} = \left(\begin{array}{ccc}
	{S_{\bm u}^E}^{-1} & \alpha (\BuE)^T& 0 \\
	0 & {S_{p}^E}^{-1} & -\tau B_{\bm w} \\
	0 & 0 & S_{\bm w}^{-1}
	\end{array}\right)^{-1}.
\end{equation}
Parameter robustness is obtained for the block upper
triangular preconditioners using the following theorems.  The proofs
are similar in concept to the proofs for Theorem~\ref{thm:BL_elim}
and~\ref{thm:BL_elim_inexact} and are, therefore, omitted.
\begin{theorem} 
	Assuming a shape regular mesh and the discretization described above,
	then there exist constants $\Sigma$ and $\Upsilon$, independent of discretization and physical parameters,
	such that, for any $\bm{x}' = (\mathcal{B}_U^E)^{-1}\bm{x}$ with $\bm{x}=(\bm{u},p,\bm{w})^T\neq\bm{0}$,
	\[
	\Sigma\leq\frac{\left(\mathcal{A}^E\mathcal{B}_U^E\bm{x}',\bm{x}'\right)_{(\mathcal{B}_D^E)}}{\left(\bm{x}',\bm{x}'\right)_{(\mathcal{B}_D^E)}},\
	\frac{\|\mathcal{A}^E\mathcal{B}_U^E\bm{x}'\|_{(\mathcal{B}_D^E)}}{\|\bm{x}'\|_{(\mathcal{B}_D^E)}}\leq\Upsilon.
	\]
\end{theorem}

\begin{theorem} 
	Assuming (\ref{eqn:equivalentHw}), (\ref{eqn:equivalentHuE}), and (\ref{eqn:equivalentHpE}) hold,
	$\| I-\AuE S_{\bm u}^E\|_{A_{\bm u}}\leq\rho$ and $\| I-\ApEs S_{p}^E\|_{\ApEs}\leq\beta$ with $\rho>0$ and $\beta>0$ sufficiently small,
	there exist constants $\Sigma$ and $\Upsilon$, independent of discretization and physical parameters,
	such that, for any $\bm{x}' = (\mathcal{B}_U^E)^{-1}\bm{x}$ with $\bm{x}=(\bm{u},p,\bm{w})^T\neq\bm{0}$,
	\[
	\Sigma\leq\frac{\left(\mathcal{A}^E\widehat{\mathcal{B}_U^E}\bm{x}',\bm{x}'\right)_{(\widehat{\mathcal{B}_D^E})}}
		{\left(\bm{x}',\bm{x}'\right)_{(\widehat{\mathcal{B}_D^E})}},\
	\frac{\|\mathcal{A}^E\widehat{\mathcal{B}_U^E}\bm{x}'\|_{(\widehat{\mathcal{B}_D^E})}}
		{\|\bm{x}'\|_{(\widehat{\mathcal{B}_D^E})}}\leq\Upsilon.
	\]
\end{theorem}

This shows that the constructed block preconditioners are robust with respect to the physical and discretization parameters of the bubble-eliminated system, (\ref{block_form_elim}).


\section{Numerical Results}\label{sec:num}
In this section, we illustrate the convergence benefits obtained using
the preconditioners presented above. All test problems were
implemented in the HAZmath library
\cite{2014AdlerJ_HuX_ZikatanovL-aa}, which contains routines for
finite elements, multilevel solvers, and graph
algorithms.
The numerical tests were performed on a workstation with an 8-core 3GHz Intel Xeon ``Sandy Bridge'' CPU and 32 GB of RAM per core.

For each test we use flexible GMRES to solve the linear system
obtained from both
the bubble-enriched P1-RT0-P0 discretization, $\mathcal{A}$
\eqref{block_form}, and the bubble-eliminated discretization,
$\mathcal{A}^E$ \eqref{block_form_elim}. A
stopping tolerance of $10^{-8}$ was used for the relative residual of the linear system\edit{, measured relative to the norm of the right hand side}.
For the discretization parameters, tests cover different mesh sizes and different time step sizes.
To show robustness with respect to the physical parameters, the permeability, $\bm{K}$,
and the Poisson ratio $\nu$ are varied.  \edit{We also consider a 3D test problem where there are jumps in the permeability.}
In all test cases we consider a diagonal permeability tensor $\bm{K} = k\bm{I}$.
The exact solves for the blocks in $\mathcal{B}_D$, $\mathcal{B}_L$,
and $\mathcal{B}_U$ are done using the
UMFPACK library \cite{TADavis2,TADavis1,TADavis4,TADavis3}.  For the
inexact blocks, $S_{\bm u}$ and $S_{\bm u}^{E}$ are inverted using GMRES preconditioned with unsmoothed aggregation AMG in a V-cycle, solved to a relative residual tolerance of $10^{-3}$.
The $S_{\bm w}$ block is solved using an auxiliary space preconditioned GMRES to a relative residual tolerance of $10^{-3}$ \cite{Arnold2000,HiptmairXu2007,KolevVassilevski}.
Using a piecewise constant finite-element space for pressure results in a diagonal matrix for $M_p$, so the action of $S_p$ is directly computed in the full bubble case. In the bubble-eliminated case, $S_p^E$ is inverted using GMRES preconditioned with unsmoothed aggregation AMG in a V-cycle, solved to a relative residual tolerance of $10^{-3}$.

\subsection{Two-Dimensional Test Problem}
First, we consider the Mandel problem in two-dimensions, which models an infinitely long saturated porous slab sandwiched between a top and bottom rigid frictionless plate, and is an important benchmarking tool as the analytical solution is known \cite{Abousleiman1996}. At time $t=0$, each plate is loaded with a constant vertical force of magnitude $2F$ per unit length as shown in Figure~\ref{fig:mandel}. The analytical solution for pressure is given by
\begin{equation}
	p(x,y,t) = 2 p_0 \sum_{n=1}^{\infty} \frac{\sin{\alpha_n}}{\alpha_n - \sin{\alpha_n}\cos{\alpha_n}}\left(\cos{\frac{\alpha_n x}{a}} - \cos{\alpha_n}\right)
	\exp\left(\frac{-\alpha_n^2 c t}{a^2}\right),
\end{equation}
where $p_0 = \frac{1}{3a}B(1+\nu_u)F$, $B=1$ is the Skempton's coefficient, $\nu_u=\frac{3\nu+B(1-2\nu)}{3-B(1-2\nu)}$ is the undrained Poisson ratio,
$c$ is the consolidation coefficient given by $c=K(\lambda+2\mu)$, and $\alpha_n$ are the positive roots to the nonlinear equation,
\[\tan\alpha_n = \frac{1-\nu}{\nu_u-\nu}\alpha_n.\]
Due to symmetry of the problem we only need to solve in the top right quadrant, defined as $\Omega = (0,1)\times(0,1)$.
We cover $\Omega$ with a uniform triangular grid by dividing an $N\times N$ uniform square grid into right triangles, where the mesh spacing is defined by $h=\frac{1}{N}$.
For the material properties, $\mu_f=1$, $\alpha=1$, and $M=10^6$, the Lam\'e coefficients are computed in terms of the Young modulus, $E=10^4$, and the Poisson ratio, $\nu$: $\lambda = \frac{E \nu}{(1-2\nu)(1+\nu)}$ and $\mu = \frac{E}{1+2\nu}$.

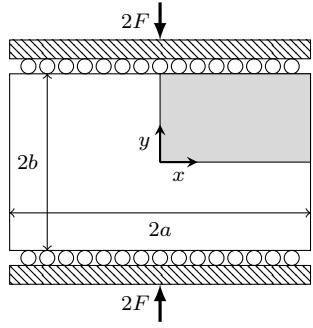
\begin{figure}[h!]
\begin{center}
\begin{tikzpicture}
	\def\xLen{4}
	\def\yLen{3}
	\def\bdryH{0.25} 
	\def\circR{0.1} 
	\draw[pattern=north west lines] (0,\yLen) rectangle (\xLen,\yLen+\bdryH);
	\draw[pattern=north west lines] (0,0) rectangle (\xLen,\bdryH);
	\foreach \circCenter in {0.25,0.5,...,3.75}
	{
		\draw (\circCenter,\yLen-\circR) circle (\circR);
		\draw (\circCenter,\bdryH+\circR) circle (\circR);
	}
	\draw (0,\bdryH+2*\circR) rectangle (\xLen,\yLen-2*\circR);
	\draw[fill=gray!30] (\xLen/2,\yLen/2+\bdryH/2) rectangle (\xLen,\yLen-2*\circR);
	\def\axesLen{0.5}
	\draw[thick,-stealth] (\xLen/2,\yLen/2+\bdryH/2) -- (\xLen/2 + \axesLen,\yLen/2+\bdryH/2) node[midway,below] {\footnotesize $x$};
	\draw[thick,-stealth] (\xLen/2,\yLen/2+\bdryH/2) -- (\xLen/2, \yLen/2+\bdryH/2 + \axesLen) node[midway,left] {\footnotesize $y$};
	\def\dPos{0.5}
	\draw[<->] (0,\dPos+\bdryH+2*\circR) -- (\xLen,\dPos+\bdryH+2*\circR) node[midway,below] {\footnotesize $2a$};
	\draw[<->] (\dPos,\bdryH+2*\circR) -- (\dPos,\yLen-2*\circR) node[midway,left] {\footnotesize $2b$};
	\def\FLen{0.5}
	\draw[very thick, -latex] (\xLen/2,-\FLen) -- (\xLen/2, 0) node[midway,left] {\footnotesize $2F$};
	\draw[very thick, -latex] (\xLen/2,\yLen+\bdryH+\FLen) -- (\xLen/2, \yLen+\bdryH) node[midway,left] {\footnotesize $2F$};
\end{tikzpicture}
\end{center}
\caption{2D physical and computational domain for Mandel's problem.}
\label{fig:mandel}
\end{figure}

Table~\ref{tab:block-prec-full-2d-ht} shows iterations counts for the block preconditioners on the full bubble system for different mesh sizes and time-step sizes.  Here, we take one time step using Backward Euler.
The physical parameters used in these tests were $\nu=0.0$ and $k=10^{-6}$.
We see from the relatively consistent iteration counts that the preconditioned system is robust with respect to the discretization parameters.
The block upper and lower triangular preconditioners contain more coupling information than the block diagonal preconditioners, and as a result we see that they preform better than the block diagonal preconditioners.
\begin{table}[h!]
	\setlength{\tabcolsep}{3pt}
	\footnotesize
	\begin{center}
		\caption{Full bubble system. Iteration counts for the block preconditioners on the 2D Mandel problem with varying discretization parameters.}\label{tab:block-prec-full-2d-ht}
		\begin{tabular}{|l || c c c c c|}
			\hline 
			&\multicolumn{5}{ c| }{$\mathcal{B}_D$}\\ \hline
			\backslashbox{$\tau$}{$h$} & $\frac{1}{8}$ & $ \frac{1}{16}$ & $\frac{1}{32}$  & $\frac{1}{64}$ & $\frac{1}{128}$  \\
			\hline
			$0.1$    & 39 & 40 & 40 & 40 & 38 \\
			$0.01$   & 26 & 34 & 39 & 39 & 38 \\
			$0.001$  & 23 & 23 & 28 & 34 & 37 \\
			$0.0001$ & 21 & 21 & 21 & 21 & 21 \\
			\hline 
		\end{tabular}
		\begin{tabular}{|c c c c c|}
			\hline
			\multicolumn{5}{ |c| }{$\mathcal{B}_L$}\\ \hline
			$\frac{1}{8}$ \hspace{-50pt}\phantom{\backslashbox{$\tau$}{Mesh}}  & $\frac{1}{16}$ & $\frac{1}{32}$  & $\frac{1}{64}$ & $\frac{1}{128}$ \\
			\hline
			19 & 19 & 18 & 17 & 17 \\
			15 & 18 & 19 & 18 & 17 \\
			11  & 12  & 15  & 17 & 18 \\
			11  & 10  & 10  & 13  & 15  \\
			\hline 
		\end{tabular}
		\begin{tabular}{|c c c c c|}
			\hline
			\multicolumn{5}{ |c| }{$\mathcal{B}_U$}\\ \hline
			$\frac{1}{8}$ \hspace{-50pt}\phantom{\backslashbox{$\tau$}{Mesh}}  & $\frac{1}{16}$ & $\frac{1}{32}$  & $\frac{1}{64}$ & $\frac{1}{128}$ \\
			\hline
			19 & 19 & 19 & 18 & 17 \\
			14  & 17 & 18 & 18 & 17 \\
			10  & 11  & 14  & 17 & 17 \\
			8  & 9  & 9  & 12  & 14  \\
			\hline 
		\end{tabular}
		\linebreak
		\begin{tabular}{|l ||c c c c c|}
			\hline
            &\multicolumn{5}{ c| }{$\widehat{\mathcal{B}_D}$}\\ \hline
            \backslashbox{$\tau$}{$h$} & $\frac{1}{8}$ & $\frac{1}{16}$ & $\frac{1}{32}$  & $\frac{1}{64}$ & $\frac{1}{128}$ \\
			\hline
			$0.1$    & 39 & 40 & 40 & 40 & 36 \\
			$0.01$   & 26 & 34 & 39 & 39 & 38 \\
			$0.001$  & 23 & 23 & 23 & 34 & 37 \\
			$0.0001$ & 21 & 22 & 21 & 23 & 29 \\
			\hline 
		\end{tabular}
		\begin{tabular}{|c c c c c|}
			\hline
			\multicolumn{5}{ |c| }{$\widehat{\mathcal{B}_L}$}\\ \hline
			$\frac{1}{8}$ \hspace{-50pt}\phantom{\backslashbox{$\tau$}{Mesh}}  & $\frac{1}{16}$ & $\frac{1}{32}$  & $\frac{1}{64}$ & $\frac{1}{128}$ \\
			\hline
			19 & 20 & 19 & 19 & 18 \\
			15 & 18 & 19 & 19 & 18 \\
			11  & 13  & 15 & 17 & 18 \\
			11  & 11  & 11  & 13  & 15  \\
			\hline 
		\end{tabular}
		\begin{tabular}{|c c c c c|}
			\hline
			\multicolumn{5}{ |c| }{$\widehat{\mathcal{B}_U}$}\\ \hline
			$\frac{1}{8}$ \hspace{-50pt}\phantom{\backslashbox{$\tau$}{Mesh}}  & $\frac{1}{16}$ & $\frac{1}{32}$  & $\frac{1}{64}$ & $\frac{1}{128}$ \\
			\hline
			19 & 19 & 19 & 18 & 20 \\
			14 & 17 & 18 & 18 & 17 \\
			10 & 12 & 15  & 17 & 17 \\
			9  & 9  & 10  & 12  & 15  \\
			\hline 
		\end{tabular}
	\end{center}
\end{table}%


Similar observations are made for Table~\ref{tab:block-prec-elim-2d-ht}, which shows iteration counts for the block preconditioners on the bubble-eliminated system for different mesh sizes and time-step sizes.
We see that using the bubble-eliminated system results in a slight degradation in performance, but nothing significant.
It is also important to note that the performance impact of using the inexact block preconditioners is negligible versus using the exact block preconditioners.
This implies that the inexact preconditioners could potentially be solved with less strict tolerance, resulting in more computational efficiency.

\begin{table}[h!]
	\setlength{\tabcolsep}{3pt}
	\footnotesize
	\begin{center}
		\caption{Bubble-eliminated system. Iteration counts for the block preconditioners on the 2D Mandel problem with varying discretization parameters}\label{tab:block-prec-elim-2d-ht}
		\begin{tabular}{|l || c c c c c|}
			\hline 
			&\multicolumn{5}{ c| }{$\mathcal{B}_D^E$}\\ \hline
			\backslashbox{$\tau$}{$h$} & $\frac{1}{8}$ & $ \frac{1}{16}$ & $\frac{1}{32}$  & $\frac{1}{64}$ & $\frac{1}{128}$  \\
			\hline
			$0.1$    & 36 & 40 & 43 & 43 & 42 \\
			$0.01$   & 26 & 30 & 37 & 40 & 40 \\
			$0.001$  & 32 & 29 & 25 & 31 & 35 \\
			$0.0001$ & 34 & 35 & 31 & 25 & 26 \\
			\hline 
		\end{tabular}
		\begin{tabular}{|c c c c c|}
			\hline
			\multicolumn{5}{ |c| }{$\mathcal{B}_L^E$}\\ \hline
			$\frac{1}{8}$ \hspace{-50pt}\phantom{\backslashbox{$\tau$}{Mesh}}  & $\frac{1}{16}$ & $\frac{1}{32}$  & $\frac{1}{64}$ & $\frac{1}{128}$ \\
			\hline
			23 & 23 & 23 & 22 & 21 \\
			17 & 21 & 22 & 22 & 22 \\
			17 & 15 & 18 & 21 & 22 \\
			19 & 18 & 16 & 14 & 18 \\
			\hline 
		\end{tabular}
		\begin{tabular}{|c c c c c|}
			\hline
			\multicolumn{5}{ |c| }{$\mathcal{B}_U^E$}\\ \hline
			$\frac{1}{8}$ \hspace{-50pt}\phantom{\backslashbox{$\tau$}{Mesh}}  & $\frac{1}{16}$ & $\frac{1}{32}$  & $\frac{1}{64}$ & $\frac{1}{128}$ \\
			\hline
			22 & 23 & 23 & 22 & 21 \\
			16 & 20 & 22 & 22 & 21 \\
			14  & 14 & 16 & 20 & 21 \\
			14 &14  &14 & 13 & 17 \\
			\hline 
		\end{tabular}
        \linebreak
		\begin{tabular}{|l ||c c c c c|}
			\hline
            &\multicolumn{5}{ c| }{$\widehat{\mathcal{B}_D}^E$}\\ \hline
            \backslashbox{$\tau$}{$h$} & $\frac{1}{8}$ & $\frac{1}{16}$ & $\frac{1}{32}$  & $\frac{1}{64}$ & $\frac{1}{128}$ \\
			\hline
			$0.1$    & 36 & 40 & 43 & 43 & 43 \\
			$0.01$   & 26 & 30 & 37 & 40 & 40 \\
			$0.001$  & 32 & 29 & 25 & 31 & 35 \\
			$0.0001$ & 34 & 35 & 31 & 25 & 26 \\
			\hline 
		\end{tabular}
		\begin{tabular}{|c c c c c|}
			\hline
			\multicolumn{5}{ |c| }{$\widehat{\mathcal{B}_L}^E$}\\ \hline
			$\frac{1}{8}$ \hspace{-50pt}\phantom{\backslashbox{$\tau$}{Mesh}}  & $\frac{1}{16}$ & $\frac{1}{32}$  & $\frac{1}{64}$ & $\frac{1}{128}$ \\
			\hline
			23 & 24 & 23 & 22 & 23 \\
			17 & 21 & 22 & 23 & 22 \\
			18 & 15 & 18 & 21 & 22 \\
			19 & 18 & 16 & 15 & 18 \\
			\hline 
		\end{tabular}
		\begin{tabular}{|c c c c c|}
			\hline
			\multicolumn{5}{ |c| }{$\widehat{\mathcal{B}_U}^E$}\\ \hline
			$\frac{1}{8}$ \hspace{-50pt}\phantom{\backslashbox{$\tau$}{Mesh}}  & $\frac{1}{16}$ & $\frac{1}{32}$  & $\frac{1}{64}$ & $\frac{1}{128}$ \\
			\hline
			22 & 23 & 23 & 22 & 21 \\
			16 & 20 & 22 & 22 & 21 \\
			15 & 14 & 17 & 20 & 21 \\
			14 & 14 & 14 & 14 & 17 \\
			\hline 
		\end{tabular}
	\end{center}
\end{table}%

Table~\ref{tab:K-nu-full-2d} and Table~\ref{tab:K-nu-elim-2d} show iteration counts for the block preconditioners when the physical values of $\nu$ and $K$ are varied for the full bubble system and bubble-eliminated system. The mesh size is fixed to $h=\frac{1}{128}$, and the time-step size is $\tau=0.01$.
Again, we observe robustness, this time with respect to the physical parameters.
The use of inexact preconditioners and the bubble elimination have minimal impact on the performance.
In the limit of impermeability ($k \rightarrow 0$), or in the limit of the Poisson ratio approaching $0.5$,
the three-field Biot model limits to the Stokes' Equation.
An interesting result is the better performance when the system is approaching this case.



\begin{table}[htp]
\footnotesize
\begin{center}
\caption{Full bubble system. Iteration counts for the block preconditioners on the 2D Mandel problem with varying physical parameters $K$ and $\nu$.} \label{tab:K-nu-full-2d}
\begin{tabular}{|c|c c c c c c|}
  \hline
  & \multicolumn{6}{ c| }{$\nu = 0.0$ and varying $K$} \\ \hline
  & $1$ & $10^{-2}$ & $10^{-4}$ & $10^{-6}$ & $10^{-8}$ & $10^{-10}$  \\ \hline
  $\mathcal{B}_D$           & 23  & 25 & 35 & 38 & 29 & 19 \\
  $\mathcal{B}_L$           & 7  & 11  & 15 & 17 & 15  & 9  \\
  $\mathcal{B}_U$           & 13  & 16 & 17 & 16 & 15  & 7  \\
  $\widehat{\mathcal{B}_D}$ & 35  & 33 & 36 & 38 & 29 & 19 \\
  $\widehat{\mathcal{B}_L}$ & 14  & 15 & 16 & 18 & 15  & 10  \\
  $\widehat{\mathcal{B}_U}$ & 27  & 22 & 17 & 17 & 15 & 8 \\
  \hline
\end{tabular}
\begin{tabular}{|c c c c c c|}
  \hline
  \multicolumn{6}{ |c| }{$K = 10^{-6}$ and varying $\nu$} \\ \hline
  $0.1$ & $0.2$ & $0.4$ & $0.45$ & $0.49$ & $0.499$  \\ \hline
  45  & 52 & 39 & 36 & 28 & 20 \\[.0em]
  16  & 19 & 11 & 11 & 9 & 10  \\[.01em]
  20  & 22 & 16 & 14 & 11 & 16 \\[.05em]
  45  & 52 & 39 & 26 & 23 & 17 \\[.2em]
  17  & 20 & 14 & 12 & 11 & 12  \\[.2em]
  21  & 24 & 17 & 16 & 10 & 16 \\
  \hline
\end{tabular}
\end{center}
\end{table}%


\begin{table}[htp]
\footnotesize
\begin{center}
\caption{Bubble-eliminated system. Iteration counts for the block preconditioners on the 2D Mandel problem with varying physical parameters $K$ and $\nu$.} \label{tab:K-nu-elim-2d}
\begin{tabular}{|c|c c c c c c|}
  \hline
  & \multicolumn{6}{ c| }{$\nu = 0.0$ and varying $K$} \\ \hline
  & $1$ & $10^{-2}$ & $10^{-4}$ & $10^{-6}$ & $10^{-8}$ & $10^{-10}$  \\ \hline
  $\mathcal{B}_D^E$               & 36  & 36 & 41 & 42 & 26 & 34 \\
  $\mathcal{B}_L^E$               & 17  & 17 & 19 & 21 & 18 & 16  \\
  $\mathcal{B}_U^E$                & 23  & 22 & 22 & 21 & 17 & 12  \\
  $\widehat{\mathcal{B}_D}^E$ & 36  & 38 & 41 & 43 & 26 & 34 \\
  $\widehat{\mathcal{B}_L}^E$ & 20  & 20 & 20 & 23 & 18 & 17  \\
  $\widehat{\mathcal{B}_U}^E$ & 27  & 27 & 22 & 21 & 17 & 13  \\
  \hline
\end{tabular}
\begin{tabular}{|c c c c c c|}
  \hline
  \multicolumn{6}{ |c| }{$K = 10^{-6}$ and varying $\nu$} \\ \hline
  $0.1$ & $0.2$ & $0.4$ & $0.45$ & $0.49$ & $0.499$  \\ \hline
  43  & 54 & 44 & 43 & 39 & 22 \\[.12em]
  20  & 24 & 21 & 20 & 17 & 12  \\[.12em]
  24  & 28 & 23 & 23 & 20 & 17 \\[.28em]
  43  & 54 & 44 & 43 & 39 & 20 \\[.25em]
  20  & 26 & 22 & 21 & 18 & 13 \\[.25em]
  25  & 28 & 23 & 23 & 20 & 17 \\
  \hline
\end{tabular}
\end{center}
\end{table}%

Finally, Figure~\ref{fig:time2D} shows the time scaling with respect to mesh size
for the three different inexact preconditioners for the full-bubble and bubble-eliminated systems.
The timings scale on the order of $O(N\log N)$ which is nearly optimal.
We also see that while a single iteration of the block lower or block upper triangular preconditioner will take longer
than that of a block diagonal iteration, the fewer required iterations of the block triangular preconditioners results in
a net savings in total computational time.
The bubble-eliminated system, being a smaller system than the full-bubble system, takes less time to solve. Figure~\ref{fig:time2D} shows that solving the bubble-eliminated system is nearly ten times faster than solving the full-bubble system.

\begin{figure}[ht!]
    \centering
    \includegraphics[scale=0.75]{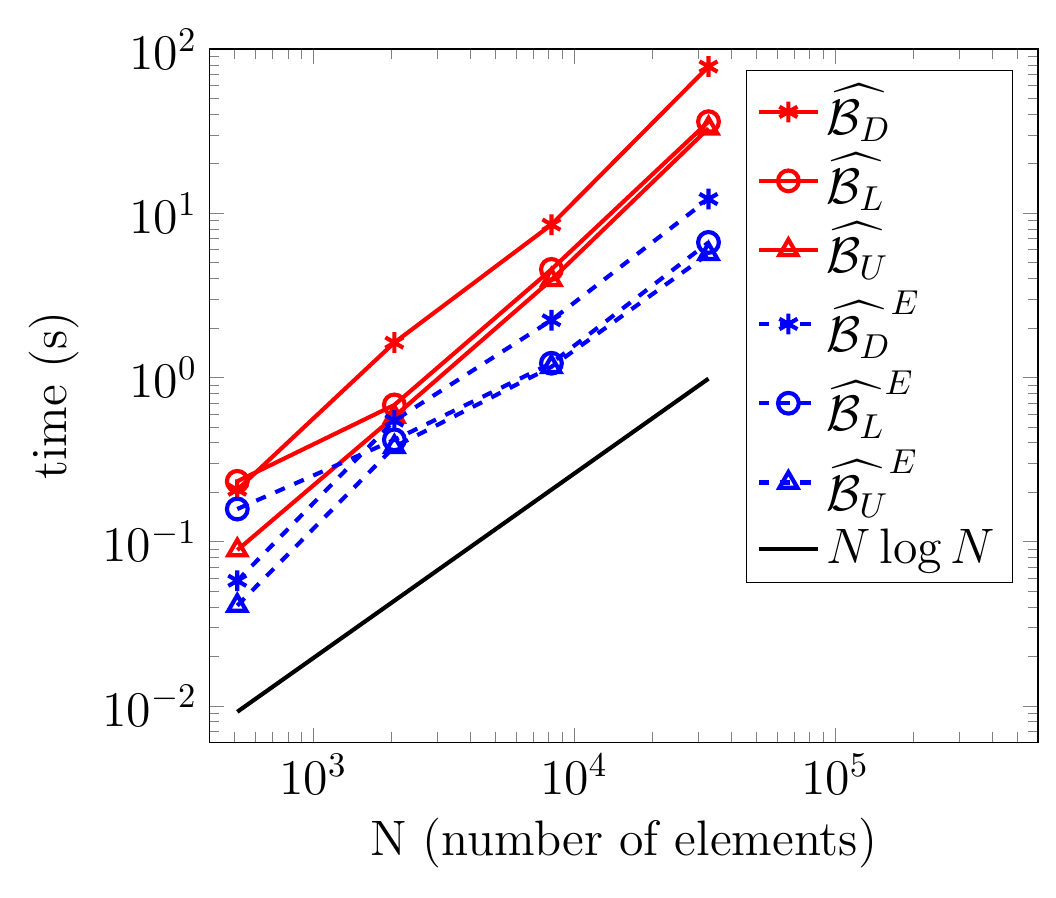}
    \caption{Timing results versus mesh size for the full bubble and bubble-eliminated system for the 2D Mandel problem, where $N$ is the total number of elements. The performance comparison between the inexact block diagonal, block upper triangular and block lower triangular
        preconditioners is shown.}
    \label{fig:time2D}
\end{figure}

\subsection{Three-Dimensional Test Problem}
Next, we consider a footing problem in three-dimensions as seen in \cite{NLA:NLA587}. The domain is a unit cube modeling a block of porous soil. A uniform load, $\bm{\sigma_0}$, of intensity $3\times 10^{4}$ per unit area is applied in a square of size $0.5\times 0.5$ in the middle of the top face. The base of the domain is assumed to be fixed while the rest of the domain is free to drain.
\edit{The material properties used are $\mu_f=1$, $\alpha=1$, and $M=10^6$, the Lam\'e coefficients are computed in terms of the Young modulus and the Poisson ratio as in the 2D problem.}

\begin{figure}[h!]
    \centering
    \includegraphics[scale=1.0]{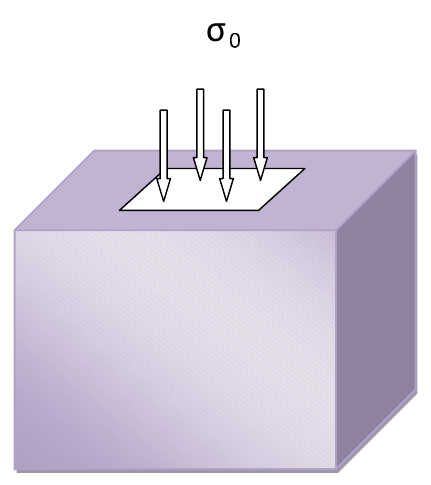}
    \includegraphics[scale=0.15]{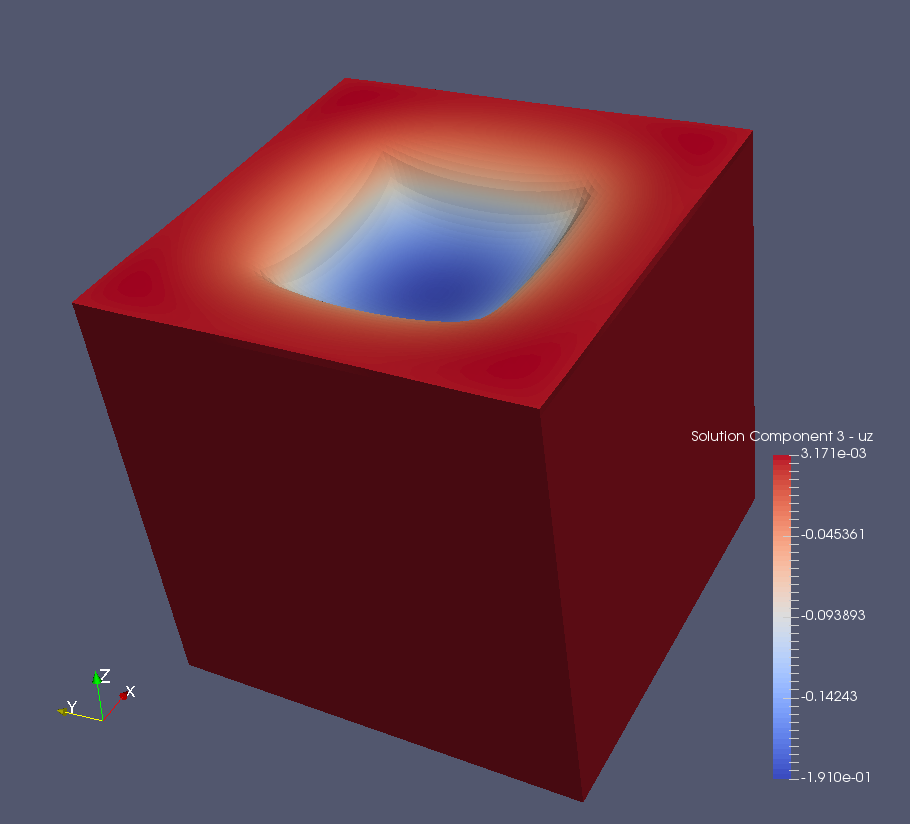}
    \caption{The three-dimensional footing problem. The image on the left shows the computational domain, while the figure on the right shows an example solution \cite{Adler2017}.}
    \label{fig:footing}
\end{figure}

Table~\ref{tab:block-prec-full-3d-ht} and Table~\ref{tab:block-prec-elim-3d-ht} show iteration counts for the block preconditioners on both systems, while varying the discretization parameters, mesh size and time-step size.  Again, one step of Backward Euler is used to test the preconditioners.
The physical parameters for these tables were $\nu=0.2$ and $k=10^{-6}$.
Here the benefits of the inexact preconditioners becomes clear, as the exact preconditioners could not be used on the two largest meshes due to memory limitations. The iteration counts confirm that the preconditioned system is robust with respect to the discretization parameters even in three dimensions.

\begin{table}[h!]
	\footnotesize
	\begin{center}
		\caption{Full bubble system. Iteration counts for the block preconditioners on the 3D footing problem with varying discretization parameters. ($*$ means the direct method for solving diagonal blocks is out of memory).}\label{tab:block-prec-full-3d-ht}
		\begin{tabular}{|l || c c c c|}
			\hline 
			&\multicolumn{4}{ c| }{$\mathcal{B}_D$}\\ \hline
			\backslashbox{$\tau$}{$h$} & $\frac{1}{4}$ & $ \frac{1}{8}$ & $\frac{1}{16}$  & $\frac{1}{32}$  \\
			\hline
			$0.1$    & 60 & 65 & 65 & $*$ \\
			$0.01$   & 47 & 57 & 68 & $*$ \\
			$0.001$  & 40 & 42 & 49 & $*$ \\
			$0.0001$ & 40 & 42 & 42 & $*$ \\
			\hline 
		\end{tabular}
		\begin{tabular}{|c c c c|}
			\hline
			\multicolumn{4}{ |c| }{$\mathcal{B}_L$}\\ \hline
			$\frac{1}{4}$ \hspace{-58pt}\phantom{\backslashbox{$\tau$}{Mesh}}  & $\frac{1}{8}$ & $\frac{1}{16}$  & $\frac{1}{32}$ \\
			\hline
			34 & 36 & 36 & $*$ \\
			30 & 34 & 37 & $*$ \\
			26 & 28 & 32 & $*$ \\
			24 & 35 & 36 & $*$ \\
			\hline 
		\end{tabular}
		\begin{tabular}{|c c c c|}
			\hline
			\multicolumn{4}{ |c| }{$\mathcal{B}_U$}\\ \hline
			$\frac{1}{4}$ \hspace{-58pt}\phantom{\backslashbox{$\tau$}{Mesh}}  & $\frac{1}{8}$ & $\frac{1}{16}$  & $\frac{1}{32}$ \\
			\hline
			32 & 34 & 34 & $*$ \\
			26 & 31 & 35 & $*$ \\
			20 & 23 & 28 & $*$ \\
			20 & 20 & 21 & $*$ \\
			\hline 
		\end{tabular}
        \linebreak
		\begin{tabular}{|l ||c c c c|}
			\hline
            &\multicolumn{4}{ c| }{$\widehat{\mathcal{B}_D}$}\\ \hline
            \backslashbox{$\tau$}{$h$} & $\frac{1}{4}$ & $\frac{1}{8}$ & $\frac{1}{16}$  & $\frac{1}{32}$ \\
			\hline
			$0.1$    & 60 & 65 & 66 & 64 \\
			$0.01$   & 47 & 58 & 68 & 71 \\
			$0.001$  & 42 & 42 & 51 & 63 \\
			$0.0001$ & 40 & 42 & 42 & 45 \\
			\hline 
		\end{tabular}
		\begin{tabular}{|c c c c|}
			\hline
			\multicolumn{4}{ |c| }{$\widehat{\mathcal{B}_L}$}\\ \hline
			$\frac{1}{4}$ \hspace{-58pt}\phantom{\backslashbox{$\tau$}{Mesh}}  & $\frac{1}{8}$ & $\frac{1}{16}$  & $\frac{1}{32}$ \\
			\hline
			34 & 36 & 36 & 36 \\
			30 & 34 & 37 & 39 \\
			26 & 28 & 32 & 36 \\
			24 & 25 & 27 & 29 \\
			\hline 
		\end{tabular}
		\begin{tabular}{|c c c c|}
			\hline
			\multicolumn{4}{ |c| }{$\widehat{\mathcal{B}_U}$}\\ \hline
			$\frac{1}{4}$ \hspace{-58pt}\phantom{\backslashbox{$\tau$}{Mesh}}  & $\frac{1}{8}$ & $\frac{1}{16}$  & $\frac{1}{32}$ \\
			\hline
			32 & 34 & 34 & 34 \\
			26 & 31 & 35 & 37 \\
			20 & 24 & 28 & 33 \\
			21 & 22 & 23 & 25 \\
			\hline 
		\end{tabular}
	\end{center}
\end{table}%


\begin{table}[h!]
	\footnotesize
	\begin{center}
		\caption{Bubble-eliminated system.  Iteration counts for the block preconditioners on the 3D footing problem with varying discretization parameters. ($*$ means the direct method for solving diagonal blocks is out of memory).}\label{tab:block-prec-elim-3d-ht}
		\begin{tabular}{|l || c c c c|}
			\hline 
			&\multicolumn{4}{ c| }{$\mathcal{B}_D^E$}\\ \hline
			\backslashbox{$\tau$}{$h$} & $\frac{1}{4}$ & $ \frac{1}{8}$ & $\frac{1}{16}$  & $\frac{1}{32}$ \\
			\hline
			$0.1$    & 61 & 65 & 66 & $*$ \\
			$0.01$   & 54 & 58 & 66 & $*$ \\
			$0.001$  & 58 & 58 & 53 & $*$ \\
			$0.0001$ & 59 & 61 & 60 & $*$ \\
			\hline 
		\end{tabular}
		\begin{tabular}{|c c c c|}
			\hline
			\multicolumn{4}{ |c| }{$\mathcal{B}_L^E$}\\ \hline
			$\frac{1}{4}$ \hspace{-58pt}\phantom{\backslashbox{$\tau$}{Mesh}}  & $\frac{1}{8}$ & $\frac{1}{16}$  & $\frac{1}{32}$ \\
			\hline
			41 & 41 & 39 & $*$ \\
			39 & 42 & 43 & $*$ \\
			37 & 39 & 40 & $*$ \\
			35 & 38 & 38 & $*$ \\
			\hline 
		\end{tabular}
		\begin{tabular}{|c c c c|}
			\hline
			\multicolumn{4}{ |c| }{$\mathcal{B}_U^E$}\\ \hline
			$\frac{1}{4}$ \hspace{-58pt}\phantom{\backslashbox{$\tau$}{Mesh}}  & $\frac{1}{8}$ & $\frac{1}{16}$  & $\frac{1}{32}$ \\
			\hline
			39 & 39 & 38 & $*$ \\
			33 & 39 & 41 & $*$ \\
			28 & 32 & 35 & $*$ \\
			29 & 29 & 30 & $*$ \\
			\hline 
		\end{tabular}
		\begin{tabular}{|l ||c c c c|}
			\hline
            &\multicolumn{4}{ c| }{$\widehat{\mathcal{B}_D}^E$}\\ \hline
            \backslashbox{$\tau$}{$h$} & $\frac{1}{4}$  & $\frac{1}{8}$ & $\frac{1}{16}$  & $\frac{1}{32}$ \\
			\hline
			$0.1$    & 61 & 65 & 66 & 66 \\
			$0.01$   & 54 & 58 & 66 & 70 \\
			$0.001$  & 58 & 58 & 53 & 61 \\
			$0.0001$ & 58 & 61 & 60 & 55 \\
			\hline 
		\end{tabular}
		\begin{tabular}{|c c c c|}
			\hline
			\multicolumn{4}{ |c| }{$\widehat{\mathcal{B}_L}^E$}\\ \hline
			$\frac{1}{4}$ \hspace{-58pt}\phantom{\backslashbox{$\tau$}{Mesh}}  & $\frac{1}{8}$ & $\frac{1}{16}$  & $\frac{1}{32}$ \\
			\hline
			41 & 41 & 39 & 39 \\
			39 & 42 & 43 & 43 \\
			37 & 39 & 40 & 43 \\
			35 & 38 & 38 & 38 \\
			\hline 
		\end{tabular}
		\begin{tabular}{|c c c c|}
			\hline
			\multicolumn{4}{ |c| }{$\widehat{\mathcal{B}_U}^E$}\\ \hline
			$\frac{1}{4}$ \hspace{-58pt}\phantom{\backslashbox{$\tau$}{Mesh}}  & $\frac{1}{8}$ & $\frac{1}{16}$  & $\frac{1}{32}$ \\
			\hline
			40 & 40 & 38 & 37 \\
			33 & 39 & 41 & 42 \\
			28 & 32 & 35 & 40 \\
			29 & 30 & 30 & 32 \\
			\hline 
		\end{tabular}
	\end{center}
\end{table}%


Table~\ref{tab:K-nu-full-3d} and Table~\ref{tab:K-nu-elim-3d} show the results when the physical parameters are varied.
The mesh size is fixed to $h=\frac{1}{16}$, and the time step size is $\tau=0.01$.
Again, we see that the preconditioned system is robust with respect to the physical parameters, and that the use of the
inexact preconditioners has little impact on the required iterations.
The bubble-eliminated system shows performance that is overall similar to the full bubble system.

\begin{table}[htp]
	\footnotesize
	\begin{center}
		\caption{Full bubble system.  Iteration counts for the block preconditioners on the 3D footing problem with varying physical parameters, $K$ and $\nu$.} \label{tab:K-nu-full-3d}
		\begin{tabular}{|c|c c c c c c|}
			\hline
			& \multicolumn{6}{ c| }{$\nu = 0.2$ and varying $K$} \\ \hline
			& $1$ & $10^{-2}$ & $10^{-4}$ & $10^{-6}$ & $10^{-8}$ & $10^{-10}$  \\ \hline
			$\mathcal{B}_D$ & 28 & 28 & 49 & 68 & 42 & 35 \\
			$\mathcal{B}_L$ & 20 & 20 & 27 & 37 & 26 & 24 \\
			$\mathcal{B}_U$ & 18 & 18 & 26 & 35 & 21 & 14 \\
			$\widehat{\mathcal{B}_D}$ & 28 & 28 & 49 & 68 & 42 & 42 \\
			$\widehat{\mathcal{B}_L}$ & 20 & 20 & 28 & 37 & 27 & 25 \\
			$\widehat{\mathcal{B}_U}$ & 21 & 21 & 27 & 35 & 22 & 24 \\
			\hline
		\end{tabular}
		\begin{tabular}{|c c c c c c|}
			\hline
			\multicolumn{6}{ |c| }{$K = 10^{-6}$ and varying $\nu$} \\ \hline
			$0.1$ & $0.2$ & $0.4$ & $0.45$ & $0.49$ & $0.499$  \\ \hline
			72 & 68 & 51 & 46 & 35 & 26 \\[.09em]
			41 & 37 & 25 & 21 & 17 & 20 \\[.09em]
			38 & 35 & 25 & 21 & 17 & 20 \\[.09em]
			72 & 68 & 51 & 46 & 35 & 26 \\[.10em]
			41 & 37 & 25 & 21 & 17 & 20 \\[.10em]
			38 & 35 & 25 & 21 & 17 & 21 \\[.10em]
			\hline
		\end{tabular}
	\end{center}
\end{table}%


\begin{table}[ht!]
	\footnotesize
	\begin{center}
		\caption{Bubble-eliminated system.  Iteration counts for the block preconditioners on the 3D footing problem with varying physical parameters, $K$ and $\nu$.} \label{tab:K-nu-elim-3d}
		\begin{tabular}{|c|c c c c c c|}
			\hline
			& \multicolumn{6}{ c| }{$\nu = 0.2$ and varying $K$} \\ \hline
			& $1$ & $10^{-2}$ & $10^{-4}$ & $10^{-6}$ & $10^{-8}$ & $10^{-10}$  \\ \hline
			$\mathcal{B}_D^E$           & 33 & 33 & 51 & 66 & 60 & 61 \\
			$\mathcal{B}_L^E$           & 20 & 20 & 29 & 43 & 38 & 35 \\
			$\mathcal{B}_U^E$           & 20 & 20 & 29 & 41 & 28 & 18 \\
			$\widehat{\mathcal{B}_D}^E$ & 33 & 33 & 51 & 66 & 60 & 61 \\
			$\widehat{\mathcal{B}_L}^E$ & 22 & 22 & 30 & 43 & 38 & 36 \\
			$\widehat{\mathcal{B}_U}^E$ & 22 & 22 & 29 & 41 & 29 & 29 \\
			\hline
		\end{tabular}
		\begin{tabular}{|c c c c c c|}
			\hline
			\multicolumn{6}{ |c| }{$K = 10^{-6}$ and varying $\nu$} \\ \hline
			$0.1$ & $0.2$ & $0.4$ & $0.45$ & $0.49$ & $0.499$  \\ \hline
			70 & 66 & 53 & 48 & 43 & 28 \\[.12em]
			46 & 43 & 32 & 28 & 24 & 21 \\[.12em]
			44 & 41 & 31 & 27 & 24 & 21 \\[.12em]
			70 & 66 & 53 & 48 & 43 & 28 \\[.25em]
			46 & 43 & 32 & 28 & 24 & 22 \\[.25em]
			44 & 41 & 31 & 28 & 24 & 22 \\[.25em]
			\hline
		\end{tabular}
	\end{center}
\end{table}%

Similarly to Figure~\ref{fig:time2D}, Figure~\ref{fig:time3D} shows time scaling with respect to mesh size for the three different inexact preconditioners for the full-bubble and bubble-eliminated systems, again showing a nearly optimal scaling of $O(N \log N)$.
The time comparison between the three different inexact preconditioners again demonstrate that the block lower and block upper triangular preconditioners are faster than the block diagonal preconditioner despite being more expensive per iteration. Finally, we see that solving the bubble-eliminated system is faster than  solving the full-bubble system as expected.

\begin{figure}[ht!]
	\centering
	\includegraphics[scale=0.75]{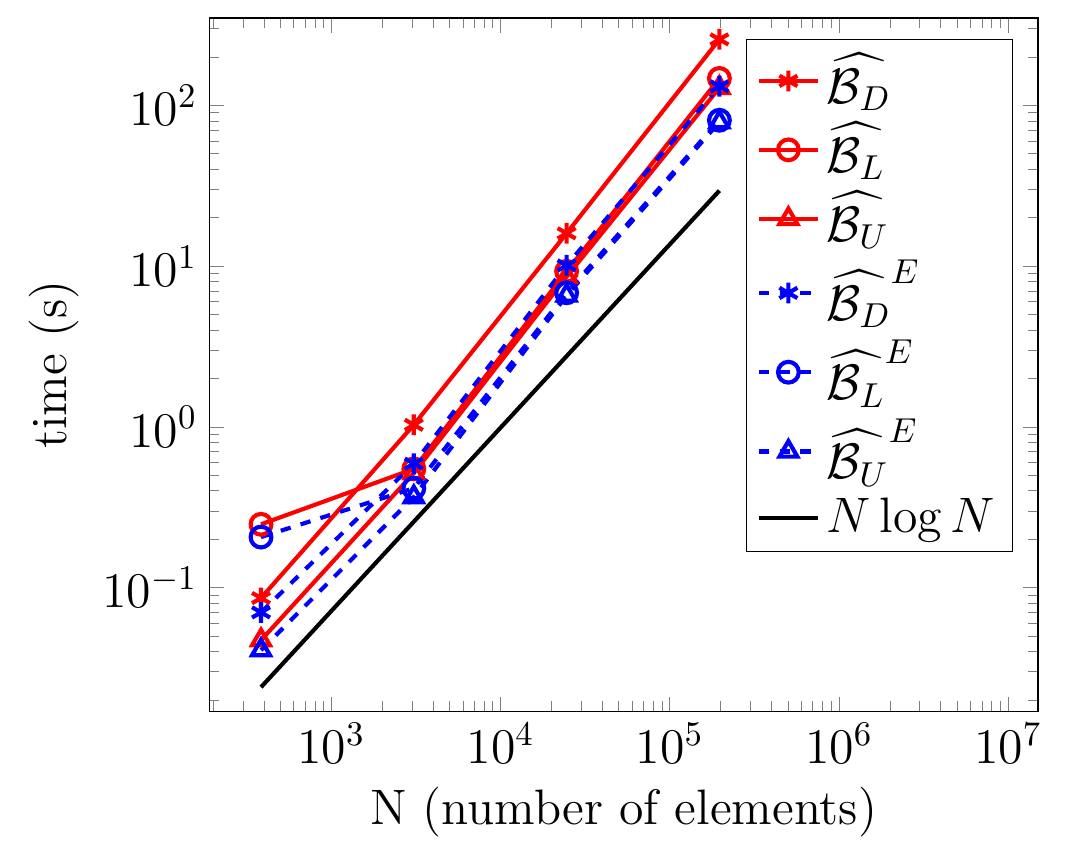}
	\caption{Timing results versus mesh size for the full bubble and bubble-eliminated systems for the 3D footing problem, where $N$ is the total number of elements.
		The performance comparison between the inexact block diagonal, block upper triangular and block lower triangular
		preconditioners is shown.}
	\label{fig:time3D}
\end{figure}

\begin{table}[ht!]
\footnotesize
\begin{center}
\caption{Full bubble system. Iteration counts for the block preconditioners on the 3D footing problem with a varying jump in the physical parameter $K$.} \label{tab:kJumpFull}
\begin{tabular}{|c|c c c c c c|}
  \hline
  & \multicolumn{6}{ c| }{$\nu = 0.2$ and $k(x)=10^{-10}$ for $x<0.5$} \\[0.5ex]
  \hline
$k(x)$ for $x\geq 0.5$& $10^{-10}$ & $10^{-8}$ & $10^{-6}$ & $10^{-4}$ & $10^{-2}$ & $1$  \\ \hline
$\mathcal{B}_D$                & 35& 42& 84& 98& 80 & 80\\[0.5ex]
$\mathcal{B}_L$                 & 24& 27& 46& 56& 51 & 51\\[0.5ex]
$\mathcal{B}_U$                & 14& 20& 38& 44& 39 & 39\\[0.5ex]
$\widehat{\mathcal{B}_D}$ & 42& 44& 84& 98& 80& 80\\[0.5ex]
$\widehat{\mathcal{B}_L}$ & 25& 28& 46& 56& 52& 51\\[0.5ex]
$\widehat{\mathcal{B}_U}$ & 24& 22& 39& 45& 44& 44\\
\hline
\end{tabular}
\end{center}
\end{table}
\begin{table}[ht!]
\footnotesize
\begin{center}
\caption{Bubble-eliminated system. Iteration counts for the block preconditioners on the 3D footing problem with a varying jump in the physical parameter $K$.} \label{tab:kJump}
\begin{tabular}{|c|c c c c c c|}
  \hline
  & \multicolumn{6}{ c| }{$\nu = 0.2$ and $k(x)=10^{-10}$ for $x<0.5$} \\[0.5ex]
  \hline
$k(x)$ for $x\geq 0.5$& $10^{-10}$ & $10^{-8}$ & $10^{-6}$ & $10^{-4}$ & $10^{-2}$ & $1$  \\ \hline
$\mathcal{B}_D^E$ &61&62&115&147&131&132\\[0.5ex]
$\mathcal{B}_L^E$ &35&39&74&84&77&78\\[0.5ex]
$\mathcal{B}_U^E$ &18&27&54&61&56&57\\[0.5ex]
$\widehat{\mathcal{B}_D}^E$ &61&62&115&147&131&133\\[0.5ex]
$\widehat{\mathcal{B}_L}^E$ &36&39&74&84&79&79\\[0.5ex]
$\widehat{\mathcal{B}_U}^E$ &29&29&55&63&61&60\\
\hline
\end{tabular}
\end{center}
\end{table}

\edit{In order to show the full capabilities of the preconditioners, we test on the 3D footing problem when there is a spatially-dependent
jump in the value for the permeability tensor $\bm{K}$.
The permeability tensor, $\bm{K}=k(x)\bm{I}$, is defined so that $k(x)=10^{-10}$ when $x<0.5$ and varied
for $x\geq0.5$.  The results are shown in Table~\ref{tab:kJumpFull} and Table~\ref{tab:kJump}.
The Poisson ratio is $\nu=0.2$, the mesh size is fixed to $h=\frac{1}{16}$, and the time-step size is $\tau=0.01$.
Note that the size of the jump increases from left to right in the table.
We see that, at the beginning, the iteration counts for the preconditioned system increases when the jump gets larger.  However, it stabilizes as the jump gets larger and, more importantly, the iterations are bounded from above.  This is consistent with our theoretical results that there is an upper bound on the condition number or field-of-values for the preconditioned system. 
}

\section{Conclusions}\label{sec:conc}

The stability and well-posedness of the discrete problem provides a
foundation for designing robust preconditioners. Thus, we are able to
develop block preconditioners which yield uniform convergence rates
for GMRES. These preconditioners are robust with respect to both the
physical and the discretization parameters, making it attractive for
problems in poromechanics, such as Biot's consolidation model
considered here. Moreover, the bubble-eliminated system has the same number of degrees of freedom as a P1-RT0-P0 discretized system,
yet it is well-posed independent of the physical and discretization parameters, and it attains performance similar to
the full bubble enriched P1-RT0-P0 system.  Due to the lower number of degrees of freedom, though, the solution time is faster than the fully-stabilized system.

Future work involves developing monolithic multigrid methods for the stabilized discretization of the three-field Biot model presented in \cite{Rodrigo2017}.
The block preconditioners presented here can then be used as a relaxation step in the monolithic multigrid method, and the overall performance will be compared against this work as stand alone preconditioners.  Additionally, other test problems including systems with fractures or other nonlinear behavior will be considered.

\appendix


\section{Proof of Theorem \ref{thm:diag_wellposed}}\label{sec:appA}
The following lemmas are useful for the following proofs.
\begin{lemma} \label{lemma:div_leq_A}
	Given the system defined in (\ref{block_form}),
	\begin{equation*}
		\| B \bm{v} \|_{M_p^{-1}} \leq \frac{1}{\zeta} \| \bm{v} \|_{A_{\bm u}},
	\end{equation*}
	where $(-\ddiv\bm{v},q)\rightarrow B$, with $q\in Q_h$ and $\bm{v}\in \bm{V}_h$.
\end{lemma}
\begin{proof}
By direct computation, 
\[
a(\bm{v}, \bm{v}) \geq \zeta^2 \| \ddiv \bm{v}\|^2 \geq \zeta^2 \| P_{Q_h} \ddiv \bm{v} \|^2,
\]
where $P_{Q_h}$ is the $L^2$-projection from $Q$ onto $Q_h$.  As an abuse of notation, we use $\bm{v}$ for the corresponding vector representation and write the above inequality in matrix form,
concluding that
\[ \| B \bm{v} \|_{M_p^{-1}} \leq \frac{1}{\zeta} \| \bm{v} \|_{A_{\bm u}}.\]
\end{proof}
\begin{corollary} \label{corollary:divb_leq_Ab}
Considering only the bubble component for Lemma \ref{lemma:div_leq_A}, we have
\[ \| B_b \bm{x}_b \|_{M_p^{-1}} \leq \| \bm{x}_b \|_{A_{bb}}^2 \leq \| \bm{x}_b \|_{D_{bb}}^2.
\]
\end{corollary}
\begin{proof}
The first inequality follows the same arguments as the proof of Lemma \ref{lemma:div_leq_A} and the second inequality comes from the spectral equivalence of $A_{bb}$ and $D_{bb}$, i.e., \eqref{eqn:ADequiv}.
\end{proof}

With the above result, we now show the well-posedness of the system given by Theorem \ref{thm:diag_wellposed}, restated here.
\begin{theorem}
If $(\bm{V}_h, \bm{W}_h, Q_h)$ is Stokes-Biot stable, then:
\begin{align}
& \sup_{\bm{0} \neq \bm{x}_h \in \bm{X}_h} \sup_{\bm{0} \neq \bm{y}_h \in \bm{X}_h } \frac{ (\mathcal{A}^D \bm{x}_h, \bm{y}_h)}{ \|\bm{x}_h\|_{\mathcal{D}} \|\bm{y}_h\|_{\mathcal{D}}} \leq \tilde{\varsigma},\\ 
& \inf_{\bm{0} \neq \bm{y}_h \in \bm{X}_h} \sup_{\bm{0} \neq \bm{x}_h \in \bm{X}_h} \frac{ (\mathcal{A}^D \bm{x}_h, \bm{y}_h)}{ \|\bm{x}_h\|_{\mathcal{D}} \|\bm{y}_h\|_{\mathcal{D}} } \geq \tilde{\gamma}, 
\end{align}
where,
\begin{equation*} 
\mathcal{D}  =
\begin{pmatrix}
  D_{bb}	& A_{bl}	& 0		& 0 \\
  A_{bl}^T	& A_{ll}	& 0		& 0 \\
  0		& 0 		& \left(\frac{\alpha^2}{\zeta^2} + \frac{1}{M}\right) M_p 	& 0 \\
  0 		& 0		& 0			& \tau M_{\bm w} + \tau^2 c_p A_{\bm w}
\end{pmatrix},
\end{equation*}
and $A_{\bm w} := B_{\bm w}^T M_p^{-1} B_{\bm w}$.
\end{theorem}

\begin{proof}
From \eqref{eqn:stokesA-inf-sup}, for a given $p \in Q_h$, there exists $\bm{z} \in \bm{V}_h$, such that\\
$(p, B_{\bm{u}} \bm{z}) \geq \frac{\gamma_B}{\eta\zeta} \| p \|_{M_p}^2$ and $\| \bm{z} \|_{A_{\bm{u}}^D} = \| p \|_{M_p}$.  Let $\bm{v} = \bm{u} - \theta_1 \bm{z}$, $\bm{r} = \bm{w}$, and\\
$q = -p - \theta_2 \tau \ddiv \bm{w}$, then, by the Cauchy-Shwarz and Young's inequality,
\begin{align*}
( \mathcal{A}^D(\bm{u}, \bm{w}, p^T), (\bm{v}, \bm{r}, q)^T ) &= \| \bm{u} \|_{A_{\bm{u}}^D}^2 - \theta_1 (A_{\bm{u}}^D \bm{u}, \bm{z}) + \theta_1 \alpha(p,  B_{\bm{u}} \bm{z}) + \tau \| \bm{w} \|^2_{M_{\bm w}}  \\
	& \quad + \frac{1}{M} \| p \|_{M_p}^2 + \theta_2 \tau \frac{1}{M}(p,  B_{\bm{w}} \bm{w}) + \theta_2 \alpha \tau ( B_{\bm{u}} \bm{u}, M_p^{-1} B_{\bm{w}} \bm{w})\\
	& \quad + \theta_2 \tau^2 \|  B_{\bm{w}} \bm{w} \|_{M_p^{-1}}^2\\
	& \geq \| \bm{u} \|_{A_{\bm{u}}^D}^2 - \frac{1}{2} \| \bm{u} \|_{A_{\bm{u}}^D}^2 - \frac{\theta_1^2}{2} \| \bm{z}\|_{A_{\bm{u}}^D}^2 + \theta_1 \frac{\alpha \gamma_B}{\eta\zeta} \| p \|_{M_p}^2 \\
	& \quad + \tau \| w \|^2_{M_{\bm w}} + \frac{1}{M} \| p \|_{M_p}^2   - \frac{3 \theta_2}{2} \frac{1}{M^2} \| p \|_{M_p}^2  \\
	& \quad - \frac{\theta_2}{6} \tau^2 \|  B_{\bm{w}} \bm{w} \|_{M_p^{-1}}^2 - \frac{\theta_2}{2} \alpha^2 \|  B_{\bm{u}} \bm{u} \|_{M_p^{-1}}^2\\
	& \quad - \frac{\theta_2}{2} \tau^2\|  B_{\bm{w}} \bm{w} \|_{M_p^{-1}}^2 + \theta_2 \tau^2 \|  B_{\bm{w}} \bm{w} \|_{M_p^{-1}}^2.
\end{align*}
Combining terms and applying Lemma~\ref{lemma:div_leq_A},
\begin{align*}
( \mathcal{A}^D(\bm{u}, \bm{w}, p^T), (\bm{v}, \bm{r}, q)^T )
	& \geq \left( \frac{1}{2} - \frac{\theta_2}{2 }\frac{\alpha^2}{\zeta^2} \right) \| \bm{u} \|_{A_{\bm{u}}^D}^2 + \tau \| \bm{w} \|^2_{M_{\bm w}} + \frac{1}{3} \theta_2 \tau^2 \|  B_{\bm{w}} \bm{w} \|_{M_p^{-1}}^2  \\
	& \quad + \left( \theta_1 \frac{\alpha \gamma_B}{\eta\zeta} - \frac{\theta^2_1}{2} \right) \| p \|_{M_p}^2 + \left( 1 - \frac{3}{4} \frac{2\theta_2}{M} \right) \frac{1}{M} \| p \|_{M_p}^2.
\end{align*}
Choosing $\theta_1 = \frac{\alpha \gamma_B}{2 \eta\zeta}$ and $\theta_2 =  \frac{1}{2}  \left( \frac{\alpha^2}{\zeta^2} + \frac{1}{M} \right)^{-1}$,
\begin{align*}
( \mathcal{A}^D(\bm{u}, \bm{w}, p^T), (\bm{v}, \bm{r}, q)^T )
	& \geq \left(\frac{1}{2}-\frac{1}{4}\right) \| \bm{u} \|_{A_{\bm{u}}^D}^2 + \tau \| \bm{w} \|^2_{M_{\bm w}} \\
	& \quad + \frac{1}{6} \tau^2 \left( \frac{\alpha^2}{\zeta^2} + \frac{1}{M} \right)^{-1} \|  B_{\bm{w}} \bm{w} \|_{M_p^{-1}}^2 + \left( \frac{3\alpha^2\gamma_B^2}{8\eta^2\zeta^2} \right) \| p \|_{M_p}^2\\
	& \quad + \left( 1 - \frac{3}{4} \right) \frac{1}{M} \| p \|_{M_p}^2 \\
	& \geq \tilde{\gamma} \| \left( \bm{u},  \bm{w}, p\right) \|_{\mathcal{D}}^2,
\end{align*}
where $\tilde{\gamma} = \min \left\{  \frac{1}{6}, \frac{3\gamma_B^2}{8\eta^2}  \right \}$.
\end{proof}

\section{Proof of Theorem \ref{thm:ElimWP}}\label{sec:appB}
To start, we use a result from \cite{BoffiBrezziFortin}, which is restated here for convenience.
\begin{proposition}[Proposition 3.4.5 in \cite{BoffiBrezziFortin}] \label{prop-SVD}
Let $B$ be an $m\times n$ matrix,
$S_X$ be an $n\times n$ symmetric positive definite matrix, and
$S_Y$ be an $m\times m$ symmetric positive definite matrix.
Define the following norms: $\| \bm{x} \|_X^2:=(S_X \bm{x})^T(S_X \bm{x})$ and $\| \bm{y} \|_Y^2:=(S_Y \bm{y})^T(S_Y \bm{y})$ for $\bm{x}\in\mathbb{R}^n,\ \bm{y}\in\mathbb{R}^m$,
and let $\beta$ be defined as
\[
	\inf_{\bm{y}\in H^T}\sup_{\bm{x}\in K^T} \frac{ (B\bm{x}, \bm{y}) }{\|\bm{x}\|_{X}\|\bm{y}\|_{Y}}=: \beta,
\]
where $K := \ker B$ and $H := \ker B^T$.
Then, $\beta$ coincides with the smallest positive singular value of the matrix $S_Y B S_X$.
\end{proposition}

With the above result, we now show the well-posedness of the bubble-eliminated system given by Theorem \ref{thm:ElimWP}, restated here.
\begin{theorem}
    If the full system (\ref{block_form}) is well-posed, satisfying \eqref{supsup-B} and \eqref{infsup-B} with respect to the norm~\eqref{weighted-norm}, then the bubble-eliminated system, \eqref{block_form_elim},
    satisfies the following inequalities for $\bm{x}^E = (\bm{u}_l, p_h, \bm{w}_h)^T \in \bm{X}^E_h$ and $\bm{y}^E = (\bm{v}_l, q_h, \bm{r}_h)^T \in \bm{X}^E_h$,
    \begin{equation} 
	\inf_{\bm{0}\neq \bm{x}^E \in \bm{X}_h^E} \sup_{\bm{0} \neq \bm{y}^E\in \bm{X}_h^E }
	\frac{(\mathcal{A}^E \bm{x}^E, \bm{y}^E)}{\|\bm{x}^E\|_{\mathcal{D}^E}\|\bm{y}^E\|_{\mathcal{D}^E}} \geq \gamma^*,
    \end{equation}
    and,
    \begin{equation} 
      \sup_{\bm{0} \neq \bm{x}^E \in \bm{X}_h^E} \sup_{\bm{0} \neq \bm{y}^E \in \bm{X}^E_h} 	\frac{(\mathcal{A}^E \bm{x}^E, \bm{y}^E)}{\| \bm{x}^E \|_{\mathcal{D}^E} \| \bm{y}^E \|_{\mathcal{D}^E}} \leq \varsigma ,
    \end{equation}
    where,
	\begin{equation*}
	  \mathcal{D}^E=
	  \left(\begin{array}{ccc}
	  A_{ll} - A_{bl}^T D_{bb}^{-1} A_{bl}		& 0				& 0 \\
	  0	 	& \alpha^2 B_b D_{bb}^{-1} B_b^T + c_p^{-1}M_p 	& 0 \\
	  0		& 0			& \tau M_{\bm w} + \tau^2 c_p A_{\bm w}
	  \end{array}\right),
	\end{equation*}
	with
	\begin{equation} 
\| \bm{x}^E \|_{\mathcal{D}^E}^2 =  (\mathcal{D}^E \bm{x}^E, \bm{x}^E).
	\end{equation}
Thus, (\ref{block_form_elim}) is well-posed with respect to the weighted norm (\ref{eqn:weighted_norm_elim}).
\end{theorem}

\begin{proof}
The matrix $\mathcal{A}^D$ given in (\ref{block_form_diag}) affords the following decomposition,
\begin{equation}\label{eqn:LSL}
	\mathcal{A}^D = L \mathcal{S} \tilde{L}^T,
\end{equation}
with
\begin{equation*}
  L^{-1} =
  \left(\begin{array}{cccc}
  I					&	0	& 0 	& 0 \\
  -A_{bl}^T D_{bb}^{-1}	&	I	& 0	& 0 \\
  \alpha B_b D_{bb}^{-1}	&	0	& I	& 0 \\
  0 					& 	0	& 0	& I
  \end{array}\right),
  \quad
    \tilde{L}^{-1} =
  \left(\begin{array}{cccc}
  I					&	0	& 0 	& 0 \\
  -A_{bl}^T D_{bb}^{-1}	&	I	& 0	& 0 \\
  -\alpha B_b D_{bb}^{-1}	&	0	& I	& 0 \\
  0 					& 	0	& 0	& I
  \end{array}\right),
\end{equation*}
and
\begin{equation*}
\mathcal{S}  =
\begin{pmatrix}
  D_{bb}&	0								& 0									 	& 0 \\
  0	& A_{ll} - A_{bl}^T D_{bb}^{-1} A_{bl}		& \alpha B_l^T - \alpha A_{bl}^T D_{bb}^{-1} B_b^T		& 0 \\
  0	& -\alpha B_l + \alpha B_{b} D_{bb}^{-1} A_{bl} 	& \alpha^2 B_b D_{bb}^{-1} B_b^T + \frac{1}{M} M_p 	& -\tau B_{\bm w} \\
  0 	& 0					& \tau B_{\bm w}^T			& \tau M_{\bm w}
\end{pmatrix}.
\end{equation*}
Note that $\mathcal{A}^E$ is a sub-matrix of $\mathcal{S}$.
Then, looking at the inf-sup condition \eqref{infsup-AD} \edit{\sout{.}}
we have, for $\bm{x} = (\bm{u}_b, \bm{u}_l, p_h, \bm{w}_h)^T \in \bm{X}_h$ and $\bm{y} = (\bm{v}_b, \bm{v}_l, q_h, \bm{r}_h)^T \in \bm{X}_h$,
\begin{align*}
\frac{(\mathcal{A}^D\bm{x},\bm{y})}{\|\bm{x}\|_{\mathcal{D}}\|\bm{y}\|_{\mathcal{D}}}
	=  \frac{(L\mathcal{S}\tilde{L}^T\bm{x},\bm{y})}{\|\bm{x}\|_{\mathcal{D}}\|\bm{y}\|_{\mathcal{D}}}
	=  \frac{(\mathcal{S}\bm{\xi},\bm{\varphi})}{\|\bm{\xi}\|_{\tilde{L}^{-1}\mathcal{D}\tilde{L}^{-T}}\|\bm{\eta}\|_{L^{-1}\mathcal{D} L^{-T}}},
\end{align*}
where $\bm{\xi}= \tilde{L}^T\bm{x}$, $\bm{\varphi} = L^T \bm{y}$.
We will proceed by showing that $L^{-1} \mathcal{D} L^{-T}$ and $\tilde{L}^{-1} \mathcal{D} \tilde{L}^{-T}$ are spectrally equivalent to the
following block diagonal matrix,
\begin{equation*}
  \tilde{\mathcal{D}}=
  \left(\begin{array}{cccc}
  D_{bb}		&	0								& 0			 	& 0 \\
  0			&	A_{ll} - A_{bl}^T D_{bb}^{-1} A_{bl}		& 0				& 0 \\
  0		 	& 	0	 	& \alpha^2 B_b D_{bb}^{-1} B_b^T + c_p^{-1}M_p 	& 0 \\
  0 			& 	0		& 0			& \tau M_{\bm w} + \tau^2 c_p A_{\bm w}
  \end{array}\right).
\end{equation*}
Note that $\mathcal{D}^E$, corresponding to the weighted norm on the bubble-eliminated system (\ref{eqn:weighted_norm_elim}), is a sub-matrix of $\tilde{\mathcal{D}}$.

By direct computation, the Cauchy-Schwarz inequality, Young's inequality, and use of Corollary~\ref{corollary:divb_leq_Ab} we have, for $\bm{x} = (\bm{u}_b, \bm{u}_l, p_h, \bm{w}_h)^T \in \bm{X}_h$,
\begin{align*}
(L^{-1} \mathcal{D} L^{-T} \bm{x}, \bm{x}) =& (\tilde{\mathcal{D}} \bm{x}, \bm{x} ) + \alpha(B_b \bm{u}_b, p) + \alpha(B_b^T p, \bm{u}_b) \\
	\geq& (\tilde{\mathcal{D}} \bm{x}, \bm{x} ) - \alpha\|B_b\bm{u}_b\|_{M_p^{-1}} \|p\|_{M_p} - \alpha\|\bm{u}_b\|_{D_{bb}} \|p\|_{B_b D_{bb}^{-1} B_b^T}\\
	\geq& (\tilde{\mathcal{D}} \bm{x}, \bm{x} ) - \frac{1}{3}\|\bm{u}_b\|_{D_{bb}}^2 - \frac{3 \alpha^2}{4\zeta^2}\|p\|_{M_p}^2
		 - \frac{1}{3}\|\bm{u}_b\|_{D_{bb}}^2 - \frac{3 \alpha^2}{4}\|p\|_{B_b D_{bb}^{-1} B_b^T}^2 \\
	=&  \frac{1}{3}\|\bm{u}_b\|_{D_{bb}}^2
	+\alpha^2\frac{1}{4}\|p\|_{B_b D_{bb}^{-1} B_b^T}^2
	+ \left(\left(\frac{\alpha^2}{\zeta^2}+\frac{1}{M}\right) - \frac{3 \alpha^2}{4\zeta^2} \right) \| p \|_{M_p}^2 \\
	&+ \|\bm{u}_{l}\|_{\AuE}^2
	+\| \bm{w} \|_{\tau M_{\bm w} + \tau^2 c_p A_{\bm w}}^2,
\end{align*}
where $\AuE = A_{ll} - A_{bl}^T D_{bb}^{-1} A_{bl}$.
Thus, we get that
\begin{equation} \label{eqn:LDLlower}
( L^{-1} \mathcal{D} L^{-T} \bm{x}, \bm{x} ) \geq \frac{1}{4}(\tilde{\mathcal{D}}\bm{x}, \bm{x}).
\end{equation}

Similarly,
\begin{align*}
(L^{-1} \mathcal{D} L^{-T} \bm{x}, \bm{x}) =& (\tilde{\mathcal{D}} \bm{x}, \bm{x} ) + \alpha(B_b \bm{u}_b, p) + \alpha(B_b^T p, \bm{u}_b) \\
	\leq& (\tilde{\mathcal{D}} \bm{x}, \bm{x} ) + \alpha\|B_b\bm{u}_b\|_{M_p^{-1}} \|p\|_{M_p} + \alpha\|\bm{u}_b\|_{D_{bb}} \|p\|_{B_b D_{bb}^{-1} B_b^T} \\
	\leq& (\tilde{\mathcal{D}} \bm{x}, \bm{x} ) + \frac{1}{2}\|\bm{u}_b\|_{D_{bb}}^2 + \frac{\alpha^2}{2\zeta^2}\|p\|_{M_p}^2
		 + \frac{1}{2}\|\bm{u}_b\|_{D_{bb}}^2 + \frac{\alpha^2}{2}\|p\|_{B_b D_{bb}^{-1} B_b^T}^2 \\
	=&  2\|\bm{u}_b\|_{D_{bb}}^2
	+\alpha^2(1+\frac{1}{2})\|p\|_{B_b D_{bb}^{-1} B_b^T}^2
	+\left(\left(\frac{\alpha^2}{\zeta^2}+\frac{1}{M}\right) + \frac{\alpha^2}{2\zeta^2} \right) \| p \|_{M_p}^2 \\
	&+ \|\bm{u}_{l}\|_{\AuE}^2
	+\| \bm{w} \|_{\tau M_{\bm w} + \tau^2 c_p A_{\bm w}}^2,
\end{align*}
yielding,
\begin{equation} \label{eqn:LDLupper}
( L^{-1} \mathcal{D} L^{-T} \bm{x}, \bm{x} ) \leq 2(\tilde{\mathcal{D}}\bm{x}, \bm{x}).
\end{equation}
With (\ref{eqn:LDLlower}) and (\ref{eqn:LDLupper}) we have that
$L^{-1} \mathcal{D} L^{-T}$ is spectrally equivalent to $\tilde{\mathcal{D}}$.  For the $\tilde{L}^{-1} \mathcal{D} \tilde{L}^{-T}$ operator, by direct computation and the Cauchy-Schwarz inequality, we have
\begin{align*}
(\tilde{L}^{-1} \mathcal{D} \tilde{L}^{-T} \bm{x}, \bm{x}) =& (\tilde{\mathcal{D}} \bm{x}, \bm{x} ) - \alpha(B_b \bm{u}_b, p) - \alpha(B_b^T p, \bm{u}_b) \\
	\geq& (\tilde{\mathcal{D}} \bm{x}, \bm{x} ) - \alpha\|B_b\bm{u}_b\|_{M_p^{-1}} \|p\|_{M_p} - \alpha\|\bm{u}_b\|_{D_{bb}} \|p\|_{B_b D_{bb}^{-1} B_b^T},
\end{align*}
and the rest of the proof for the lower bound follows exactly as it does in the $L^{-1} \mathcal{D} L^{-T}$ case. Similarly for the upper bound, we have
\begin{align*}
(\tilde{L}^{-1} \mathcal{D} \tilde{L}^{-T} \bm{x}, \bm{x}) =& (\tilde{\mathcal{D}} \bm{x}, \bm{x} ) - \alpha(B_b \bm{u}_b, p) - \alpha(B_b^T p, \bm{u}_b) \\
	\leq& (\tilde{\mathcal{D}} \bm{x}, \bm{x} ) + \alpha\|B_b\bm{u}_b\|_{M_p^{-1}} \|p\|_{M_p} + \alpha\|\bm{u}_b\|_{D_{bb}} \|p\|_{B_b D_{bb}^{-1} B_b^T},
\end{align*}
and the rest of the proof for the upper bound follows from the $L^{-1} \mathcal{D} L^{-T}$ case.
Thus $L^{-1} \mathcal{D} L^{-T}$ and $\tilde{L}^{-1} \mathcal{D} \tilde{L}^{-T}$ are spectrally equivalent to the
block diagonal matrix $\tilde{\mathcal{D}}$. We then write, $\forall \bm{x}, \bm{y}$,
\begin{equation*}
	\frac{(\mathcal{A}^D\bm{x},\bm{y})}{\|\bm{x}\|_{\mathcal{D}}\|\bm{y}\|_{\mathcal{D}}}
	= \frac{(\mathcal{S}\bm{\xi},\bm{\varphi})}{\|\bm{\xi}\|_{\tilde{L}^{-1} \mathcal{D} \tilde{L}^{-T}}\|\bm{\varphi}\|_{L^{-1}\mathcal{D} L^{-T}}}
	\leq \frac{16(\mathcal{S}\bm{\xi},\bm{\varphi})}{\|\bm{\xi}\|_{\tilde{\mathcal{D}}}\|\bm{\varphi}\|_{\tilde{\mathcal{D}}}}.
\end{equation*}
Since the maps $\tilde{L}^T: \bm{x}\mapsto\bm{\xi}$ and $L^T: \bm{y}\mapsto\bm{\varphi}$ are one-to-one,
\begin{equation*}
	\inf_{\bm{0} \neq \bm{\xi}\in \bm{X}_h} \sup_{\bm{0}\neq \bm{\varphi}\in \bm{X}_h}
	\frac{(\mathcal{S}\bm{\xi},\bm{\varphi})}{\|\bm{\xi}\|_{\tilde{\mathcal{D}}}\|\bm{\varphi}\|_{\tilde{\mathcal{D}}}} \geq \gamma^*,
\end{equation*}
where $\gamma^* = \frac{\tilde{\gamma}}{16}$.

Evoking Proposition~\ref{prop-SVD} (Proposition 3.4.5 in \cite{BoffiBrezziFortin}),
we know that the smallest singular value of
$\tilde{\mathcal{D}}^{-1/2}\mathcal{S}\tilde{\mathcal{D}}^{-1/2}$ is bounded from below by a fixed positive constant.
The matrix,
\begin{equation*}
\tilde{\mathcal{D}}^{-1/2}\mathcal{S}\tilde{\mathcal{D}}^{-1/2} =
\begin{pmatrix} D_{bb}^{-1/2} D_{bb} D_{bb}^{-1/2} & 0 \\ 0 & (\mathcal{D}^E)^{-1/2} \mathcal{A}^E (\mathcal{D}^E)^{-1/2} \end{pmatrix},
\end{equation*}
 is a block diagonal matrix with
$(\mathcal{D}^E)^{-1/2} \mathcal{A}^E (\mathcal{D}^E)^{-1/2} $ as a submatrix on the diagonal.
Then, the smallest singular value of $(\mathcal{D}^E)^{-1/2} \mathcal{A}^E (\mathcal{D}^E)^{-1/2} $ is bounded from below by a fixed positive constant. Therefore, we arrive at equation (\ref{infsup-E}), for $\bm{x}^E = (\bm{u}_l, p_h, \bm{w}_h)^T \in \bm{X}^E_h$ and $\bm{y}^E = (\bm{v}_l, q_h, \bm{r}_h)^T \in \bm{X}^E_h$,
\begin{equation*}
	\inf_{\bm{0}\neq \bm{x}^E \in \bm{X}_h^E} \sup_{\bm{0}\neq \bm{y}^E\in \bm{X}_h^E}
	\frac{(\mathcal{A}^E \bm{x}^E, \bm{y}^E)}{\|\bm{x}^E\|_{\mathcal{D}^E}\|\bm{y}^E\|_{\mathcal{D}^E}} \geq \gamma^*.
\end{equation*}

The upper bound follows from the following set of inequalities,
\begin{align*}
	&\sup_{\bm{0}\neq \bm{x}\in \bm{X}_h}\sup_{\bm{0} \neq \bm{y} \in \bm{X}_h }
		\frac{(\mathcal{A}^D\bm{x},\bm{y})}{\|\bm{x}\|_{\mathcal{D}}\|\bm{y}\|_{\mathcal{D}}}
	\\
	\geq&  \sup_{\bm{0} \neq \bm{\xi}\in \bm{X}_h}\sup_{\bm{0}\neq \bm{\varphi} \in \bm{X}_h }
		\frac{(\mathcal{S}\bm{\xi},\bm{\varphi})}{4 \|\bm{\xi}\|_{\tilde{\mathcal{D}}}\|\bm{\varphi}\|_{\tilde{\mathcal{D}}}}\\
	\geq & \sup_{ \bm{0}\neq \bm{x}^E \in  \bm{X}_h^E }\sup_{ \bm{0}\neq \bm{y}^E \in \bm{X}_h^E}
	\frac{
	\left(	\begin{pmatrix}
			D_{bb}	&	0	\\
			0		&	\mathcal{A}^E
		\end{pmatrix} \begin{bmatrix} 0 \\ \bm{x}^E \end{bmatrix}, \begin{bmatrix} 0 \\ \bm{y}^E \end{bmatrix}\right)
	}{
		4 \left\|\begin{bmatrix} 0 \\ \bm{x}^E \end{bmatrix}\right\|_{\tilde{\mathcal{D}}}
		\left\|\begin{bmatrix} 0 \\ \bm{y}^E \end{bmatrix}\right\|_{\tilde{\mathcal{D}}}
	}\\
	= & \sup_{ \bm{0}\neq \bm{x}^E\in \bm{X}_h^E }\sup_{\bm{0}\neq \bm{y}^E \in \bm{X}_h^E}
		\frac{(\mathcal{A}^E \bm{x}^E, \bm{y}^E)}{ 4\|\bm{x}^E\|_{\mathcal{D}^E}\|\bm{y}^E\|_{\mathcal{D}^E}},
\end{align*}
which results in (\ref{continuity-E}).
\end{proof}

\bibliographystyle{siamplain}
\bibliography{references}

\end{document}